\documentclass[11pt,reqno]{amsart}

\usepackage{graphicx,color,xcolor} 
\usepackage{amssymb,amsmath,amsthm,amsfonts,mathrsfs}
\usepackage{fan}
\usepackage{hyperref}

\title{Even and odd minimal categorifications of the nilpotent part of classical and quantum $\mathsf{sl}(2)$}
\date{September 2, 2026}

\author{Mikhail Khovanov}
\address{Department of Mathematics, Johns Hopkins University, Baltimore, MD 21218, USA}
\email{\href{mailto:khovanov@jhu.edu}{khovanov@jhu.edu}}

\author{Alvaro L. Martinez} 
\address{Department of Mathematics, Columbia University, New York, NY 10027, USA}
\email{\href{mailto:alm2297@columbia.edu}{alm2297@columbia.edu}}

\author{Fan Zhou} 
\address{Department of Mathematics, Columbia University, New York, NY 10027, USA}
\email{\href{mailto:fz2326@columbia.edu}{fz2326@columbia.edu}}

\def\R{\mathbb R}
\def\Q{\mathbb Q}
\def\Z{\mathbb Z}
\newcommand{\kk}{\mathbbm{k}}
\newcommand{\NH}{\mathsf{NH}}
\newcommand{\ONH}{\mathsf{ONH}}
\newcommand{\NC}{\mathsf{NC}}
\newcommand{\ONC}{\mathsf{ONC}}
\newcommand{\wNC}{\widehat{\NC}}
\newcommand{\wONC}{\widehat{\ONC}}
\newcommand{\Sym}{\on{Sym}}
\newcommand{\slt}{\mathsf{sl}_2}

\newcommand{\SPol}{\mathsf{SPol}}
\newcommand{\lra}{\longrightarrow}
\newcommand{\End}{\operatorname{End}}
\newcommand{\Hom}{\operatorname{Hom}}
\newcommand{\Mat}{\operatorname{Mat}}
\newcommand{\Kar}{\mathsf{Kar}}
\newcommand{\mcC}{\mathcal{C}}
\newcommand{\mcE}{\mathcal{E}}
\newcommand{\mcNH}{\mathcal{NH}}
\newcommand{\mcNC}{\mathcal{NC}}

\newcommand{\mfS}{\mathfrak{S}}
\newcommand{\opartial}{\overline{\partial}}

\newcommand{\gdim}{\on{{gdim}}}

\theoremstyle{definition}
\newtheorem{thm}{Theorem}[section]
\newtheorem{lem}[thm]{Lemma}
\newtheorem{remark}[thm]{Remark}
\newtheorem{cor}[thm]{Corollary}
\newtheorem{prop}[thm]{Proposition}
\newtheorem{example}[thm]{Example}

\theoremstyle{definition}
\newtheorem{definition}[thm]{Definition}

\tikzset{every picture/.style={line width=0.9pt}} 
\newcommand{\npartial}{{%
  \mathpalette\raisenot\partial
}}

\makeatletter
\newcommand{\raisenot}[2]{%
  \raise.4\fontdimen22
    \ifx#1\displaystyle
      \textfont2
    \else
      \ifx#1\textstyle
        \textfont2
      \else
        \ifx#1\scriptstyle
          \scriptfont2
        \else
          \scriptscriptfont2
        \fi
      \fi
    \fi
    \rlap{%
    \settowidth\dimen@{$\m@th#1{#2}$}%
    \kern.5\dimen@
    \settowidth\dimen@{$\m@th#1=$}%
    \kern-.5\dimen@
    $\m@th#1\not$%
  }%
  {#2}%
}
\makeatother 
\nc{\npd}{\npartial}
\NewEnviron{diagram}%
  {\mathord{\hackcenter{\begin{tikzpicture}[scale=0.375]
    \BODY
\end{tikzpicture}}
  }}
\nc{\wcalNC}{\wh{\mathcal{NC}}}
\NewEnviron{raisediagram}[1][]{
	\raisebox{#1}{$\mathord{\hackcenter{\begin{tikzpicture}[scale=0.375]
    \BODY
\end{tikzpicture}}
  }$}
}

\usetikzlibrary{decorations.pathmorphing}

\tikzcdset{
  red squiggle/.style={
    red,
    line width=0.5pt,
    decorate,
    decoration={snake, amplitude=1.5pt, segment length=8pt}
  }
}
\tikzcdset{
  blue squiggle/.style={
    blue,
    line width=0.5pt,
    decorate,
    decoration={snake, amplitude=1.5pt, segment length=5pt}
  }
}
\tikzcdset{
  purple squiggle/.style={
    purple,
    line width=0.5pt,
    decorate,
    decoration={snake, amplitude=1.5pt, segment length=3pt}
  }
}

\nc{\OLbd}{\mathsf{O}\Lambda}
\nc{\osp}{\mathfrak{osp}}
\nc{\pgl}{\mathfrak{pgl}}
\renewcommand{\sl}{\mathsf{sl}}
\pgfplotsset{compat=1.18}

\begin{document}

\begin{abstract}
We write down generators and relations for a minimalist version of the nilHecke category which allows us to categorify the divided powers of a generating functor. We then classify these minimalist counterparts of the nilHecke category, subject to suitable assumptions. A minimalist version of the odd nilHecke category is introduced as well. 
\end{abstract}

\maketitle
\tableofcontents

%
%

\section{Introduction}

When does an additive endofunctor $\mcE$ acting on an additive category $\mcC$ admit divided power functors? By this we mean endofunctors $\mcE^{(n)}$  on $\mcC$ such that there is an isomorphism 
\begin{equation}\label{eq_iso}
\mcE^n \cong n!\, \mcE^{(n)},
\end{equation} 
i.e., the $n$-th power of $\mcE$ is the sum of $n!$ copies of some functor $\mcE^{(n)}$ for each $n\ge 2$?  By taking endomorphisms of both sides of \eqref{eq_iso}, existence of $\mcE^{(n)}$ implies that the endomorphism ring of $\mcE^{n}$ is  isomorphic to a matrix algebra of size $n!\times n!$ over some ring $R_n$: 
\begin{equation}\label{eq_endo_n}
\End(\mcE^{n}) \cong \Mat_{n!}(R_n). 
\end{equation}
It is convenient to require that $\mcC$ is Karoubi-closed. Then, having matrix algebra decompositions \eqref{eq_endo_n} is both necessary and sufficient for the existence of $\mcE^{(n)}$. 

A well-known solution to this problem~\cite{Lau10,rouquier20082,KL09} realizes $\End(\mcE^n)$ as the nilHecke algebra $\NH_n$. This is the algebra of endomorphisms of the vector space of polynomials $\kk[x_1,\dots, x_n]$ generated by operators of multiplication by $x_1,\dots, x_n$ and the Demazure-Bernstein-Gelfand-Gelfand divided difference operators $\partial_1,\dots, \partial_{n-1}$, see~\cite{BGG73,Dem73}. Here $\kk$ is the ground field. 

The nilHecke algebra is isomorphic to the matrix algebra, 
\[
\NH_n \cong \Mat_{n!}(\Sym_n),
\]
of size $n!$ over the ring $\Sym_n\cong \kk[x_1,\dots, x_n]^{S_n}$ of symmetric functions in $n$ variables. Natural inclusions of algebras 
\[
\NH_n\otimes \NH_m \subset \NH_{m+n}
\]
correspond to homomorphisms of rings of natural transformations
\[
\End(\mcE^n) \otimes \End(\mcE^m) \lra \End(\mcE^{n+m}). 
\]
The present paper shows a more economical way to categorify divided powers of a functor and reduce from the coefficients in $\Sym_n$ to coefficients in the ground field. The generating object for this minimalist categorification is denoted $\mcE_0$,  with 
\[
\End(\mcE_0^n) \cong \Mat_{n!}(\kk).
\]
A natural subalgebra  of $\NH_n$ isomorphic to $\Mat_{n!}(\kk)$ implicitly and explicitly appears in several papers on categorification~\cite{KLMS12,KQ15,EllisQi16}. In Section~\ref{sec_minimal} we give several equivalent descriptions of this subalgebra, which we call \emph{the double nilCoxeter algebra} and denote $\wNC_n$. We produce a generators and relations presentation  for the corresponding monoidal category  $\mcNC'$  of endomorphisms of $\mcE_0^n$,  over all $n$, see Section~\ref{subsection_minimal} and Proposition~\ref{prop_wNC_def_rel}.  Category $\mcNC'$ and its Karoubi envelope $\mcNC$ provide what may be called a \emph{minimal categorification} of the ring of divided powers 
\begin{equation}\label{eq_ZE}
    \Z\{E\} := \Z[E^{(n)}]_{n\ge 1}.
\end{equation} 
Ring $\Z\{E\}$ is defined as a subring of 
$\Q[E]$ with the $\Z$-basis of elements $E^{(n)}=E^n/n!$, over all $n\ge 0$. A suitable graded version of the category $\mcNC$ provides a categorification of the ring $\Z_q\{E\}$ of quantum divided powers, see Remark~\ref{rmk_graded} and Section~\ref{sec_classify}. 

\vspace{0.07in}

 In Section~\ref{sec_classify} we study
the extent to which such graded minimal categorifications are unique. We
give a counterexample to several natural uniqueness hypotheses and prove
a three-strand rigidity theorem which characterizes the twisted categories
$\mcNC_\nu'$ once the mixed three-strand structure is fixed.

\vspace{0.07in}

Section~\ref{section-odd} provides another minimal categorification of the divided powers ring, by using a natural matrix subalgebra $\Mat_{n!}(\kk)$ of the odd counterpart $\ONH_n$ of the nilHecke algebra~\cite{EKL12}. The latter algebra plays a fundamental role in the odd categorification of quantum $\slt$ in~\cite{EllisLauda16}, see also related papers~\cite{BrKl,ELV,BE,kang2013supercategorification}. It is unknown whether odd categorifications exist for quantum deformations of positive halves of the universal enveloping algebras of simple Lie algebras $\mathsf{sl}_n$ for $n>2$.  

\vspace{0.07in}

Harman-Snowden-Snyder's Delannoy category~\cite{HSS} is a rigid symmetric monoidal semisimple category built out of the Borel-Moore Euler characteristic for suitable subspaces of $\R^n$. 
It has a symmetric monoidal semisimple subcategory, denoted here $\mcC_+$, with a single generating object $L_{\bullet}$ and simple objects $L_{\bullet^n}$ parametrized by $n\in \Z_+$, with their images in the Grothendieck ring $K_0(\mcC_+)$ given by $\binom{H}{n}$. 
Category $\mcC_+$ may be called the \emph{positive Delannoy category}~\cite{KhSn}. The Grothendieck ring of $\mcC_+$ is naturally isomorphic to the ring $S$ of integer-valued polynomials $f\in \Q[H]$ with $f(k)\in \Z$ for all $k\in \Z$:
\begin{equation}
    K_0(\mcC_+) \ \cong \ S, \ \ [L_{\bullet^n}]\leftrightarrow \binom{H}{n}. 
\end{equation}
(Subring $S$ of $\Q[H]$ has a $\Z$-basis of binomials $\binom{H}{n}$, $n\ge 0$.)  Positive Delannoy category provides a  categorification of $S$, which may also be called a \emph{minimal categorification} of $S$.  
Another categorification of $S$, via the abelian monoidal nonsemisimple category of a direct limit of categories of coherent sheaves on $\mathbb{CP}^n$, as $n\to \infty$, can be found in~\cite{KhCoh2}.  

It may be interesting to compare our semisimple categorification of the ring of divided powers with these two categorifications of the ring $S$ of shifted divided powers. 

\vspace{0.07in}

We conclude the introduction with several open questions. Kostant integral form $U_{\Z}(\slt)$ of the universal enveloping algebra of the Lie algebra $\slt=\langle E,F,H\rangle$ is generated by divided powers (shifted divided powers for $H$) 
\[
E^{(n)}=\frac{E^n}{n!}, \ \  F^{(n)}=\frac{F^n}{n!}, \ \ H^{[n]}:=\binom{H}{n}= \frac{H(H-1)\dots (H-n+1)}{n!}
\]
with a basis 
\[
E^{(n_1)} H^{[n_2]} F^{(n_3)},  \ \ n_1,n_2,n_3\in \Z_+,
\]
see~\cite[Section 26]{Hum}. 
Rings  $\Z\{E\}$ in \eqref{eq_ZE} and $S$ above  are the subrings of $U_{\Z}(\slt)$ spanned by divided powers of $E$ and shifted divided powers of $H$, respectively. 
Each of these two subrings admits a semisimple categorification, described in the present paper for $\Z\{E\}$ and in~\cite{HSS} for $S$. 

Consider the subring $U^+_{\Z}$ of $U_{\Z}(\slt)$ generated by $S$ and $\Z\{E\}$. It has a $\Z$-basis of elements $E^{(n_1)}H^{[n_2]}$, $n_1,n_2\in \Z_+$, with multiplication in this basis having nonnegative integer coefficients. It is a natural question whether there exists a semisimple monoidal category $\mathcal{U}^+$ with the Grothendieck ring $U^+_{\Z}$ and simple objects corresponding to the basis elements $E^{(n_1)} H^{[n_2]}$ above. A more restrictive question is whether such a category may contain both the category $\mcC_+$ above (the positive Delannoy category) and the category in the present paper, as the subcategories categorifying the subrings $S$ and $\Z\{E\}$ of  $U^+_{\Z}$, respectively. 

A related problem is to construct a nonsemisimple categorification of the entire $U_{\Z}(\slt)$. In such a categorification one would expect $E,F$ to lift to biadjoint functors $\mcE,\mathcal{F}$, while $H$ and its shifted divided powers would lift to self-adjoint functors. 

A conjecture for a categorification of the \emph{quantized} Cartan subalgebra of $U_q(\slt)$ can be found in~\cite{Alvaro}, see also related papers~\cite{GMSW1,GMSW2,KhCoh,KhCoh2}. 
These conjectural categorifications should be related to the existing categorifications of the Lusztig idempotented form 
$\dot{U}$ 
of quantum $\slt$ and other Kac-Moody Lie algebras in~\cite{Lau10,rouquier20082,KhLa2}, possibly via functors or 2-functors into categorifications of the idempotented forms. The goal of the above problems is to avoid having to categorify idempotented forms of quantum groups and find categorifications of the original (quantized or classical) universal enveloping algebras, without having to add categorical analogues of projections $1_{\lambda}$ onto integral weights $\lambda$. 
Positive Delannoy category~\cite{HSS,KhSn}, a minimal categorification of a nilpotent subalgebra for $\slt$ in the present paper, and constructions of~\cite{Alvaro,GMSW1,GMSW2,KhCoh2} are early indications that such categorifications may exist. 

\vspace{0.1in}

{\bf Acknowledgments:}  M.K.~would like to acknowledge partial support from Simons Collaboration Award 994328 ``New Structures in Low-Dimensional Topology''.

%
%

\section{Minimal categorification of the divided powers ring}\label{sec_minimal}


\subsection{Double nilCoxeter algebra  \texorpdfstring{$\wNC_n$}{NC}}
\label{doubled_nilCoxeter}

\quad 

\noindent 
{\it Divided difference operators and the nilCoxeter algebra.}
Let $\kk$ be the ground field. For $n\ge 0$ consider three rings
\begin{eqnarray*}\label{eq_three_rings}
& & R_n =\kk[x_1,\dots, x_n], \ \ R^{\circ}_n=\kk[x_1^{\pm 1},\dots, x_n^{\pm 1}], \\
& &  R^-_n=\kk[x_1^{-1},\dots, x_n^{-1}], \ \ R_n\subset R_n^{\circ}\supset R^-_n.
\end{eqnarray*}
Involution $\psi(x_i)=x_i^{-1}$ of $R^{\circ}_n$ exchanges $R_n$ with $R^-_n$.

Denote by $\partial_i:R_n\lra R_n$ the $i$-th divided difference operator
\[\partial_i(f)=(f- s_i f)/(x_i-x_{i+1}), \ \ 1\le i\le n-1
\]
acting on $R_n$ and $R^{\circ}_n$. Here $s_i=(i,i+1)\in S_n$ is the $i$-th elementary transposition in the symmetric group $S_n$, the latter acting on the rings above by permuting indices of $x_1,\dots, x_n$. These operators were independently introduced by I.Bernstein, I.Gelfand and S.Gelfand ~\cite{BGG73} and  M.Demazure ~\cite{Dem73}.

Divided difference operators satisfy the relations
\begin{equation}\label{eq_demazure}
\partial_i^2=0, \ \ \partial_i \partial_{i+1}\partial_i = \partial_{i+1}\partial_i\partial_{i+1}, \ \  \partial_i\partial_j = \partial_j \partial_i, \ \ |i-j|>1.
\end{equation}
Define the nilCoxeter algebra $\NC_n$ as the algebra of operators on $R_n$ generated by $\partial_1, \dots, \partial_{n-1}$. Then $\dim_{\kk}(\NC_n)=n!$ and it has a basis of elements $\partial_w:=\partial_{i_1}\dots \partial_{i_k}$, where $w=s_{i_1}\dots s_{i_k}$ is a minimal presentation of the permutation $w\in S_n$ as a product of elementary transpositions. The multiplication in this basis is
\[\partial_w\partial_{w'}=\begin{cases} \partial_{ww'} & \mathrm{if} \ \ell(ww')=\ell(w)+\ell(w'), \\
 0 & \mathrm{otherwise}.
 \end{cases}
\]
Alternatively, $\NC_n$ can be defined as the algebra of operators on $R^{\circ}$ generated by $\partial_1,\dots, \partial_{n-1}$. This makes no difference
to the defining relations. Indeed, $R^{\circ}_n=\cup_{N\ge 0} R_n\cdot (x_1\dots x_n)^{-N}$,
and multiplication by $(x_1\dots x_n)^{\pm 1}$ commutes with the divided difference operators.

Let $\Sym_n\subset R_n$ be the subalgebra of symmetric functions in $x_1, \dots, x_n$. It is isomorphic to the polynomial algebra in elementary symmetric functions $e_m^{(n)}:=e_m(x_1,\dots, x_n)$, $1\le m \le n$, defined by~\eqref{eq_elem}. 
  
Let $V_n\subset R_n$ be the vector subspace with a basis
\begin{equation}\label{eq_basis_B}
B_n \ :=\ \{x_1^{a_1}x_2^{a_2}\dots x_{n-1}^{a_{n-1}}|\, 0\le a_i \le n-i \text{ for }i=1,\ldots,n-1\}.
\end{equation}
Then $|B_n|=n!$. Consider the monomial 
\begin{equation}\label{eq_monomial}
\underline{x}_n:=x_1^{n-1}x_2^{n-2}\dots x_{n-1}.
\end{equation}
The set $B_n$ consists of all monomials whose $x_1, \dots, x_n$ degrees are at most those in $\underline{x}_n$ and 
\[
V_n = R_n\cap (R_n^- \underline{x}_n)\subset R_n^{\circ}.
\]
The following proposition is well-known. For instance, see Proposition 3.4 in ~\cite{Lau10}.

\begin{prop} $R_n$ is a free graded $\Sym_n$-module with a homogeneous basis $B_n$.
\end{prop}

Consequently, there is an isomorphism of free graded $\Sym_n$-modules $R_n\cong V_n\otimes \Sym_n$. 

Algebra $R^{\circ}_n$ is given by localizing $R_n$ at the monomial $x_1x_2\dots x_n$. Define $\Sym_n^{\circ}\subset R^{\circ}_n$ to be the subalgebra of $S_n$-invariant Laurent polynomials under the permutation action of the symmetric group. Then $\Sym_n^{\circ}$ is the localization of $\Sym_n$ given by inverting $x_1x_2\dots x_n$. Consequently $R_n^{\circ}$ is a free graded $\Sym_n^{\circ}$ module with basis $B_n$, and 
\begin{equation}  \label{eq_iso_circc}
R_n^{\circ}\cong V_n\otimes \Sym_n^{\circ}
\end{equation}
as graded free $\Sym_n^{\circ}$-modules. 

\vspace{0.07in}

{\it Negative divided difference operators.}
Define a version of the divided difference operator acting on $R_n^-$ and $R_n^{\circ}$:
\begin{equation}\label{eq_d_bar}
\opartial_i f := (f-s_i f)/(x_i^{-1}-x_{i+1}^{-1}).
\end{equation}
Clearly, $\opartial_i$ is conjugate to $\partial_i$ via $\psi$, that is, $\psi\opartial_i\psi=\partial_i$ as endomorphisms of $R_n$ and $R^{\circ}_n$. Manipulating the denominator in \eqref{eq_d_bar} one gets  
\begin{equation}\label{eq_opart_t}
    \opartial_i = - x_i x_{i+1} \partial_i .
\end{equation}

Define \emph{the negative divided difference operator} 
\begin{equation}\label{eq_d_minus}\partial_i^-:= - x_i\partial_i x_{i+1}
\end{equation}
and view it as an operator on $R_n$ and $R^{\circ}_n$.

We make $R_n,R^{\circ}_n$ and $R_n^-$ graded by $\deg(x_i)=2$, $\deg(x_i^{-1})=-2$.
Operator $\partial_i^-$ has the opposite degree to that of $\partial_i$. Namely, $\deg(\partial_i)=-2$ and $\deg(\partial_i^-)=2$.

For a polynomial $f\in R_n$ in $x_1,\dots, x_n$ there is the corresponding polynomial $\psi(f)\in R_n^-$ in $x_1^{-1},\dots, x_n^{-1}$. Likewise, $\psi$ takes a Laurent polynomial $f$ to a Laurent polynomial $\psi(f)$.

\begin{prop} \label{prop_partial_minus} For $f\in R^{\circ}_n$,
\begin{equation}
  \partial_i^{-}(\psi(f)\underline{x}_n)= \psi(\partial_i f)\underline{x}_n =
  \opartial_i(\psi(f))\underline{x}_n.
\end{equation}
\end{prop}

\begin{proof} The second equality is just the relation $\psi\partial_i = \opartial_i \psi$.
To establish the first equality,
starting with $f=f(x_1,x_2)\in R_2^\circ$ we compute
\begin{eqnarray*} - x_1 \partial_1 x_2 (x_1\psi(f)) & = & -x_1^2x_2 \partial_1 \psi(f)\\
& = & x_1 \opartial_1\psi(f)\\
& = &x_1 \psi \partial_1(f)
\end{eqnarray*}
where the first equality holds because $x_1x_2$ commutes with $\partial_1$, and for the second one we applied ~\eqref{eq_opart_t}.
Thus, $\partial_1^-(\psi(f)x_1) =\psi(\partial_1 f)x_1$. Since $\partial_1^-$ commutes with $x_1x_2$, we also have $\partial_1^-(\psi(f)x_1^{a+1}x_2^a) =\psi(\partial_1 f)x_1^{a+1}x_2^a$ for all $a\ge 0$.

Replacing $x_1,x_2,\partial_1$ by $x_i,x_{i+1}$ and $\partial_i$, respectively, implies the proposition.
\end{proof}

{\it Involution $\tau$.}
Let $\tau$ be the involution
\begin{equation}\label{eq_map_x} f\longmapsto \underline{x}_n \cdot \psi(f), \ \  f\in R^{\circ}_n.
\end{equation}
\begin{cor} \label{cor_conjugation} For $f\in R^{\circ}_n$
\[
  \tau\partial_i^{-}\tau f= \partial_i f ,
  \]
  so that 
  \[
  \partial_i^- = \tau \partial_i \tau. 
  \]
\end{cor}
This is a rewrite of the first equality in Proposition~\ref{prop_partial_minus}.

\begin{prop}\label{prop_involution_endomorphism}
    The involution $\tau$ restricts to  an endomorphism of $V_n$.
\end{prop}
\begin{proof}
    To see this, take an arbitrary basis monomial in $B_n$
    \[x_1^{a_1}x_2^{a_2}\dots x_{n-1}^{a_{n-1}}\]
    with $0\le a_i \le n-i$ for $i=1,\ldots,n-1$.
    Then \[\tau(x_1^{a_1}x_2^{a_2}\dots x_{n-1}^{a_{n-1}})=x_1^{n-1-a_1}x_2^{n-2-a_2}\dots x_{n-1}^{1-a_{n-1}}\]
    Multiplying $0\le a_i \le n-i$ by $-1$ and adding $n-i$ to each term, we get \[n-i \ge n-i-a_i \ge 0,\]so the image of the monomial lies in $V_n$.
\end{proof}

Note that both $\partial_i$ and $\partial_i^-$ preserve subspaces $R_n$ and $R_n^-  \underline{x}_n$ and their intersection $V_n$. 

\vspace{0.1in} 

{\it Negative and double nilCoxeter algebras.}
Denote by $\NC_n^-$ the subalgebra of endomorphisms of $R^{\circ}_n$ generated by
$\partial_1^-,\dots, \partial_{n-1}^-$. Due to Corollary~\ref{cor_conjugation},
there is a natural isomorphism $\NC_n^-\cong \NC_n$ taking $\partial^-_i$ to $\partial_i$.

\begin{cor}\label{cor_relations_minus}
    The following relations hold on operators in $R_n$ and $R^{\circ}_n$:
\begin{equation}
\partial^-_i\partial_i^-=0, \ \ \partial_i^- \partial_{i+1}^-\partial_i^- = \partial_{i+1}^-\partial^-_i\partial_{i+1}^-, \ \  \partial_i^-\partial_j^- = \partial_j^- \partial_i^-, \ \ |i-j|>1.
\end{equation}
\end{cor}

For $w\in S_n$ define $\partial_w^-:=\partial_{i_1}^-\dots \partial_{i_k}^-$ for a reduced presentation $w=s_{i_1}\dots s_{i_k}$.
Then $\{\partial^-_w|w\in S_n\}$ is a basis of $\NC_n^-$ and $\dim \NC_n^-=n!$.

Sometimes we denote $\partial_i$ by $\partial_i^+$ and $\NC_n$ by $\NC_n^+$.
Then, for $\varepsilon\in \{+,-\}$, there are endomorphisms $\partial_i^{\varepsilon}$ of $R_n$ and $R^{\circ}_n$, and the algebra $\NC_n^{\varepsilon}$ is generated by them. Given a reduced expression $w=s_{i_1}\cdots s_{i_r}$, we may write $\partial^{\varepsilon}_w:=\partial^{\varepsilon}_{i_1,\ldots,i_r}$.

Note that both operators $\partial^+_i,\partial^-_i$ commute with multiplication by $(x_1\dots x_n)^{\pm 1}$ and 
\[
R^{\circ}_n=\cup_{N\ge 0}(x_1\dots x_n)^{-N}R_n,
\]
so that a $\kk$-linear relation on compositions of these operators holds in $R_n$
iff it holds in $R^{\circ}_n$.

\begin{lem} \label{lemma_stable} For $N\ge 0$ the subspace $V_2^N\subset R_2$ spanned by monomials in the set
  \[ B_2^N \ := \ \{ x_1^a x_2^b | a\le N+1, b\le N\}
  \]
is stable under operators $\partial_1$ and $\partial_1^-=-x_1\partial_1 x_2$.
\end{lem}

\begin{proof}
  Since $x_1x_2\in \Sym_2$, multiplication by $x_1x_2$ commutes with operators $\partial_1,\partial_1^-$. If both $a,b>0$ we can factor out
  $x_1^ax_2^b= (x_1x_2) x_1^{a-1}x_2^{b-1}$ and reduce from $N$ to $N-1$.

$N=0$ case: $\partial_1(1)=0$, $\partial_1^-(1)= -x_1\partial_1 x_2(1)=x_1\in V_2^0$, and
$\partial_1 (x_1) = 1\in V_2^0, \partial_1^-(x_1)=0$.

For $a\le N+1$, $\partial_1(x_1^a) \in V_2^N$ since it is a sum of monomials
of total degree $a-1\le N$. For $0<a \le N+1$,
\[ \partial_1^-(x_1^a)=-x_1 \partial_1 (x_1^ax_2 )= - x_1^2x_2 \partial_1(x_1^{a-1}) = - x_1^2 x_2 \sum_{c=0}^{a-2} x_1^c x_2^{a-2-c}.
\]
Inspecting individual monomials in that sum shows that they all belong to $B_2^N$.

For $b\le N$ it is clear that $\partial_1(x_2^b)\in V_2^N$ and
\[  \partial_1^-(x_2^b) = - x_1 \partial_1 (x_2^{b+1}) = x_1 \sum_{c=0}^b x_1^c x_2^{b-c}.
\]
All monomials in this sum are in $V_2^N$.
\end{proof}
We have $V_2^N= R_2\cap (R_2^- x_1^{N+1}x_2^N)$.

\begin{cor}\label{cor_subspace_stable} The subspace $V_n\subset R_n$ is stable under operators $\partial_i$ and
  $\partial_i^-$, for $i=1,\dots, n-1$.
\end{cor}

\begin{proof} This follows at once from Lemma~\ref{lemma_stable} or from Corollary~\ref{cor_conjugation}.
\end{proof}

Let $\wNC_n$ be the subalgebra of endomorphisms of $R_n$ generated by $\partial_1,\dots, \partial_{n-1}$ and $\partial^-_1, \dots, \partial^-_{n-1}$.
It contains $\NC_n$ and $\NC_n^-$ as subalgebras. Corollary~\ref{cor_conjugation} tells us that there is an involution on $\wNC_n$ transposing $\partial_i$ and $\partial_i^-$, $1\le i\le n-1$.

\vspace{0.1in}

\begin{prop}\label{prop_nc_invol}
    $\wNC_n$ has an involution $\tau$ transposing $\partial_i$ and $\partial_i^-$, for $i=1,\dots, n-1$.
\end{prop}

\begin{proof}
    This follows at once from Corollary~\ref{cor_conjugation}.
\end{proof}

\subsection{NilHecke algebra and category}
\label{subsec_nilHecke}

The nilHecke algebra $\NH_n$ is the algebra of operators on $R_n$ generated by the divided difference operators $\partial_1, \dots, \partial_{n-1}$ and by operators of multiplication by $x_1, \dots, x_n$. It plays an important role in the geometry of the flag variety and its Schubert cells~\cite{FP98,BL00,Man01,kumar2012kac} and in the categorification of quantum $sl(2)$~\cite{CR08,Lau10,Lau12,KL09,rouquier2012quiver,rouquier20082}. $\NH_n$ has the following defining relations 
\begin{equation}\label{eq_nilHecke_def}
 \begin{array}{ll}
   \partial_i x_j = x_j\partial_i \quad \text{if $j\neq i,i+1$}, &
   \partial_i\partial_j = \partial_j\partial_i \quad \text{if $|i-j|>1$}, \\
  \partial_i^2 = 0,  &
   \partial_i\partial_{i+1}\partial_i = \partial_{i+1}\partial_i\partial_{i+1},  \\
   x_i \partial_i - \partial_i x_{i+1}=1,  &   \partial_i x_i - x_{i+1} \partial_i =1, \\
   x_i x_j =   x_j x_i . &  
  \end{array}
\end{equation}

The nilHecke algebra has a natural diagrammatic interpretation, where $x_i$ is represented by a dot on the $i$-th vertical strand counting from the left and $\partial_i$ by a crossing of the $i$-th and $(i+1)$-st strands: 

\[ \hackcenter{\begin{tikzpicture}[scale=0.375]
        \node at (-2,1) {$x_i \, = \, $}; 
        \draw (0,0)--(0,2);
        \draw[dotted] (0.8,1)--(2.2,1);
        \draw (3,0)--(3,2);
        \draw (5,0)--(5,2);
        \draw(7,0)--(7,2);
        \draw[dotted] (7.8,1)--(9.2,1);
        \draw (10,0)--(10,2);
        \fill (5,1) circle (5pt);
        \node at (5,-0.5) {\tiny $i$};
    \end{tikzpicture}}\ ,\qquad 
    \hackcenter{\begin{tikzpicture}[scale=0.375]
      \node at (-2,1) {$\pd_i \, = \, $};
        \draw (0,0)--(0,2);
        \draw[dotted] (0.8,1)--(2.2,1);
        \draw (3,0)--(3,2);
        \draw (5,0)--(7,2);
        \node at (5,-0.6) {\tiny $i$};
        \draw(7,0)--(5,2);
        \node at (7,-0.6) {\tiny $i\!+\!1$};
        \draw(9,0)--(9,2);
        \draw[dotted] (9.8,1)--(11.2,1);
        \draw (12,0)--(12,2);
    \end{tikzpicture}}\]
Defining relations \eqref{eq_nilHecke_def} have a diagrammatic interpretation that can be found in~\cite[Section~3.1]{Lau10},~\cite[Section~3.4.2]{Lau12} and~\cite[Section~2.1]{KLMS12}, for example.  
In analogy with these diagrammatics, we can introduce a diagram for the negative divided difference operator $\pd_i^-$ as well, depicting it by a crossing of the $i$-th and $(i+1)$-st strands with a small circle around the crossing: 
\begin{equation}\label{eq_circle_crossing}
    \begin{diagram}
        \draw(0,0)--(2,2);
        \draw(0,2)--(2,0);
        \draw (1,1) circle (5pt);
    \end{diagram}
    \coloneqq 
    -\begin{diagram}
        \draw(0,0)--(2,2);
        \draw(0,2)--(2,0);
        \fill (0.5,1.5) circle (5pt);
        \fill (1.5,0.5) circle (5pt);
    \end{diagram}
\end{equation}

The family of the nilHecke algebras $\NH_n$, over all $n\ge 0$, gives rise to a monoidal $\kk$-linear category $\mcNH'$ with a generating object $\mcE$ and hom spaces
\begin{equation}\label{eq_NH_cat_homs}
    \Hom_{\mcNH'}(\mcE^{\otimes n},\mcE^{\otimes m}) \ = \ \begin{cases} \NH_n & \mathrm{if} \ n=m, \\ 0 & \mathrm{otherwise}.
    \end{cases} 
\end{equation}
There are natural inclusions of algebras 
\[
\NH_n\otimes \NH_m \hookrightarrow \NH_{n+m}
\]
that define the tensor product of morphisms in $\mcNH'$. These inclusions are given by placing $n$-strand and $m$-strand diagrams representing products of generators of $\NH_n$ and $\NH_m$ in parallel, producing diagrams with $n+m$ strands that describe particular elements of $\NH_{n+m}$. 

\vspace{0.07in}

The $\kk$-linear monoidal category $\mcNH'$ can also be described by generating morphisms and relations on them. Generating morphisms are given by $x:\mcE\lra \mcE$ and $\partial:\mcE^{\otimes 2}\lra \mcE^{\otimes 2}$. 
Generator $x$ is represented by a dot on a vertical strand and generator $\partial$ by a crossing: 

\[ \hackcenter{\begin{tikzpicture}[scale=0.375]
        \node at (-2,1) {$x \, = \, $}; 
        \draw (0,0)--(0,2);
        \fill (0,1) circle (5pt);
    \end{tikzpicture}}\ \qquad \qquad
    \hackcenter{\begin{tikzpicture}[scale=0.375]
      \node at (-2,1) {$\pd \, = \, $};
        \draw (0,0)--(2,2);
        \draw(2,0)--(0,2);
    \end{tikzpicture}}
    \]
Defining relations are the relations in lines 2 and 3 of \eqref{eq_nilHecke_def} rewritten in the monoidal language: 
\begin{eqnarray*}
     \partial^2 & = & 0, \\ (\partial\otimes \id_\mcE)(\id_\mcE\otimes \partial)(\partial\otimes \id_\mcE) & = &  
    (\id_\mcE\otimes \partial)
    (\partial\otimes \id_\mcE)
    (\id_\mcE\otimes \partial), \\
    (x\otimes \id_\mcE)\partial - \partial (\id_\mcE\otimes x) & =&  1, \\
    \partial (x\otimes \id_\mcE) -
    (\id_\mcE\otimes x) \partial & = & 1. 
\end{eqnarray*}
Relations in lines 1 and 4 of \eqref{eq_nilHecke_def} correspond to  commutativity relations for far-away generating morphisms and they hold automatically in a monoidal category with the above generators.

\vspace{0.07in}

The additive Karoubi envelope 
\begin{equation}\label{eq_KarNH}
\mcNH:=\Kar(\mcNH')
\end{equation}
is given by forming formal finite sums of objects and then adding objects for idempotent endomorphisms. Pick a minimal idempotent $e_n$ in $\NH_n$ for each $n\ge 0$ and denote the object $(\mcE^{\otimes n},e_n)$ by $\mcE^{(n)}$. There is a natural choice for $e_n$ as follows, see~\cite[Section 2.2]{KLMS12}. Consider idempotents 
\begin{equation}\label{eq_id_e2}
e_2 := x_1\partial_1\in \NH_2
, \ \ 
e_{2,i}:= x_i\partial_{i}\in \NH_n, \ 
1\le i\le n-1.
\end{equation}
Idempotent $e_{2,i}$ is given by placing $e_2$ on the $i$-th and $(i+1)$-strands in the diagrammatic presentation for $\NH_n$. Define 
\begin{equation}\label{eq_idemp_en}
e_n =(e_{2,1}e_{2,2}\dots e_{2,n-1})(e_{2,1}e_{2,2}\dots e_{2,n-2})\dots (e_{2,1}e_{2,2})e_{2,1}
\end{equation}
to be the product of $e_{2,i}$'s  corresponding to a particular minimal presentation of the longest permutation $w_0$ of $\{1,\dots, n\}$ as the product of elementary transpositions,  
\[
w_0 = (12)(23)\dots (n-1,n)(12)(23)\dots (n-2,n-1)\dots (12)(23)(12). 
\]
Equivalently, 
$e_n = x_1^{n-1}x_2^{n-2}\dots x_{n-1}\partial_{w_0}$. Due to the Yang-Baxter relation $e_{2,1}e_{2,2}e_{2,1}=e_{2,2}e_{2,1}e_{2,2}$ idempotent $e_n$ can be presented as the product of $e_{2,i}$'s using any minimal presentation of $w_0$ via elementary transpositions. There is a direct sum decomposition of graded left $\NH_2$-modules $\NH_2 = \NH_2 e_2\{2\} \oplus \NH_2 e_2$ and similar for the $\NH_n$ and $e_n$ (see the definition of the grading below). 

Define the ring $\Z\{E\}$ of \emph{integer divided powers} to be the subring of $\Q[E]$ generated by $E^{(n)}:=\frac{E^n}{n!}$, over all $n\ge 0$. Ring $\Z\{E\}$ is a free $\Z$-module with a basis $1,E,E^{(2)},E^{(3)}, \dots$ The following formulas hold in $\Z\{E\}$: 
\begin{equation}
    E^n = n! E^{(n)}, \ \ E^{(n)}E^{(m)} = \binom{n+m}{n}E^{(n+m)}. 
\end{equation}
There is a natural isomorphism between the Grothendieck ring of $\mcNH$ and $\Z\{E\}$
\begin{equation}\label{eq_Groth_NH_nograding}
    K_0(\mcNH) \ \cong \ \Z\{E\} 
\end{equation}
that takes $[\mcE^{(n)}]$ to $E^{(n)}$ and $[\mcE^{\otimes n}]$ to $E^n$. 
In this sense, we say that the monoidal category $\mcNH$ categorifies the ring $\Z\{E\}$. 

\begin{remark}
Categories $\mcNH'$ and $\mcNH$ are symmetric monoidal, with the symmetric structure on  $\mcNH'$ given by the permutation morphism 
$\mcE^{\otimes n}\otimes \mcE^{\otimes m}\lra \mcE^{\otimes m}\otimes \mcE^{\otimes n}$ represented by the permutation endomorphism of $\kk[x_1,\dots, x_{n+m}]$ that transposes the first $n$ with the last $m$ indices of the $x$'s:
\[
x_i \longmapsto x_{m+i}, 1\le i \le n, \ \ x_{n+i}\longmapsto x_i, 1 \le i \le m. 
\]
The latter endomorphism is an element of $\NH_{n+m}$, since the transposition $x_1\mapsto x_2, x_2\mapsto x_1$ is given by $\partial_1 x_1 - x_1 \partial_1\in \NH_2$. Alternatively, it's in $\NH_{n+m}$ since it commutes with the multiplication by any symmetric function in $x_1,\dots, x_{n+m}$. Symmetric monoidal structure on $\mcNH'$ is inherited by its Karoubi envelope $\mcNH$. 
\end{remark}

\vspace{0.07in}

Ring $\Z\{E\}$ has a well-known $q$-deformation that we denote here $\Z_q\{E\}$. It is a $\Z[q,q^{-1}]$-subalgebra of $\Q(q)[E]$ with a basis, as a free $\Z[q,q^{-1}]$-module, given by  
\begin{equation}\label{eq_q_deform}
    E^{[n]}:= \frac{E^n}{[n]!}, \ \ [n]!:=[n]\dots [2][1], \ \ [k]:=\frac{q^k-q^{-k}}{q-q^{-1}}, 
\end{equation}
over all $n\ge 0$. 

To construct the $q$-analogue of \eqref{eq_Groth_NH_nograding} one introduces a grading on $R_n$ by $\deg(x_i)=2$. This induces a grading on homogeneous endomorphisms of $R_n$, with $\deg(\partial_i)=-2$ and the degree of the multiplication by $x_i$ operator equal $2$. Grothendieck ring of the category of finite-dimensional $\Z$-graded vector spaces is naturally isomorphic to $\Z[q,q^{-1}]$. One then works with homogeneous endomorphisms of $\mcE^{\otimes n}$ and uses grading shifts of objects $\mcE^{\otimes n}\{k\}$ to produce the graded version $\mcNH'_q$ of the nilHecke category. Homs in this category are degree $0$ homomorphisms between shifts of $\mcE^{\otimes n}$. Passing to the Karoubi envelope 
\begin{equation}\label{eq_mcNH_q}
\mcNH_q:=\Kar(\mcNH'_q)
\end{equation}
results in a monoidal category over the category of f.d. graded vector spaces and a natural isomorphism of rings 
\begin{equation}\label{eq_iso_K_q}
K_0(\mcNH_q) \cong \Z_q\{E\},
\end{equation}
see~\cite{Lau10,Lau12}. Under this isomorphism
\[
[(\mcE^{\otimes n},e_n)]= q^{\frac{n(n-1)}{2}}E^{[n]}.
\]

\subsection{Double nilCoxeter algebra as a subalgebra of nilHecke}
\label{subsec_subalgebra}

Observe that 
\[\partial_i^-= - x_i \partial_i x_{i+1}\in \NH_n,
\]
and algebras $\NC_n,\NC^-_n$ and $\wNC_n$ are subalgebras of $\NH_n$. 
Generators of $\NH_n$ commute with multiplications by symmetric functions
(elements of $\Sym_n$). Consequently, there is an injective homomorphism
\[ \NH_n \stackrel{\gamma}{\lra} \End_{\Sym_n}(R_n) \cong \End_{\Sym_n}(V_n\otimes \Sym_n)
\cong \End_{\kk}(V_n)\otimes_{\kk} \Sym_n.
\]
The first map takes $\NH_n$ into the algebra of $\Sym_n$-endomorphisms
of $R_n$. The latter is a free graded $\Sym_n$ module, with a basis given by
$B_n$. It is straightforward to show that $\gamma$ is an isomorphism,
resulting in a canonical isomorphism
\[
\NH_n \cong \End_{\kk}(V_n)\otimes_{\kk} \Sym_n
\]
between $\NH_n$ and a matrix algebra of size $n!$ with coefficients in $\Sym_n$.
(Note that $V_n$ has a preferred basis $B_n$, giving a natural identification
with the matrix algebra.)

\vspace{0.07in}

Corollary~\ref{cor_subspace_stable}  implies that the action of $\wNC_n$ on $R_n$
preserves the subspace $V_n$ and gives a natural algebra homomorphism
\begin{equation}\label{eq_hom_wNC}
\rho \ : \ \wNC_n \lra \End_{\kk}(V_n).
\end{equation}
Furthermore, $\wNC_n$ is a subalgebra of $\NH_n$,
and the action of the latter commutes with multiplication by elements of $\Sym_n$.
Since $R_n$ is a free $\Sym_n$-module generated by $V_n$, the following observation follows.
\begin{cor}\label{cor_rho_injective}
  Homomorphism $\rho$ is injective.
\end{cor}

We next look for commutation relations between $\partial_i$'s and $\partial_j^-$'s. Denote 
\[
\partial_{i_1,\dots, i_k}:= \partial_{i_1}\dots \partial_{i_k}, \ \ \partial^{-}_{i_1,\dots, i_k}:= \partial^{-}_{i_1}\dots \partial^{-}_{i_k}. 
\]
In particular, $\partial_{i,i+1,i}=\partial_{i+1,i,i+1}$ and $\partial^-_{i,i+1,i}=\partial^-_{i+1,i,i+1}$. 
\begin{prop}\label{prop_relations}
  The following relations hold:
  \begin{eqnarray}
  \label{eq_shorter}
    \partial_i \partial_j^-  & = & \partial_j^- \partial_i \ \ \mathrm{if} \ |i-j|>1, \\
    1 & = & \partial_i \partial_i^- + \partial_i^-\partial_i, \\
    \label{eq_short}
    \partial_i \partial_{i+1}^- & = & (\partial_i^- +\partial_{i+1}^-)\partial_{i+1}  + \partial_i^- \partial_i -1=\pd_i^-(\pd_{i+1} +\pd_i)+\pd^-_{i+1}\pd_{i+1}-1, \\
    \label{eq_long}
    \partial_{i+1}\partial_i^-  & = & \partial_i^-\partial_{i+1} + 
    [\partial_i^-,\partial_{i+1}^-]\,  [\partial_i,\partial_{i+1}]+2\partial_{i,i+1,i}^-\partial_{i,i+1,i}
  \end{eqnarray}
\end{prop}

\begin{proof} 
The relations are established by explicit computations in the nilHecke algebra, viewing $\wNC_n$ as a subalgebra of $\NH_n$. 
Let us check the last relation. For simplicity we carry out the computation on three strings, with dots on the three strings being $x_1=x, x_2=y, x_3=z$. Compute:
\begin{align*}
    \pd_2\pd_1^-&=-x^2\pd_2\pd_1+x\pd_2,\\
    \pd_1^-\pd_2&=-x^2\pd_1\pd_2+x\pd_2,\\
    \pd_{1,2,1}^-\pd_{1,2,1}&=x^2y\pd_{1,2,1},\\
    \pd_1^-\pd_2^-\pd_1\pd_2&=-(x^3+x^2y+xy^2)\pd_{1,2,1}+(x^2+xy)\pd_1\pd_2,\\
    \pd_2^-\pd_1^-\pd_1\pd_2&=-xy^2\pd_{1,2,1}+xy\pd_1\pd_2,\\
    \pd_1^-\pd_2^-\pd_2\pd_1&=-x^3\pd_{1,2,1}+x^2\pd_2\pd_1+xy\pd_2\pd_1, \\
    \pd_2^-\pd_1^-\pd_2\pd_1&=-x^2y\pd_{1,2,1}+xy\pd_2\pd_1,
\end{align*}
so that one can check the last equation is true directly:
\begin{align*}
    &\quad\ \pd_1^-\pd_2+\pd_1^-\pd_2^-\pd_1\pd_2-\pd_2^-\pd_1^-\pd_1\pd_2-\pd_1^-\pd_2^-\pd_2\pd_1+\pd_2^-\pd_1^-\pd_2\pd_1+2\pd_{1,2,1}^-\pd_{1,2,1}\\
    &=(-x^2\pd_1\pd_2+x\pd_2)+(-(x^3+x^2y+xy^2)\pd_{1,2,1}+(x^2+xy)\pd_1\pd_2)-(-xy^2\pd_{1,2,1}+xy\pd_1\pd_2)\\
    &\quad\quad\quad\quad -(-x^3\pd_{1,2,1}+x^2\pd_2\pd_1+xy\pd_2\pd_1)+(-x^2y\pd_{1,2,1}+xy\pd_2\pd_1)+(2x^2y\pd_{1,2,1})\\
    &=-x^2\pd_2\pd_1+x_1\pd_2=\pd_2\pd_1^-
\end{align*}
\end{proof}
These relations allow to move  $\partial_i$'s to the right of the $\partial_j^-$'s in any product of these operators. 
From Proposition~\ref{prop_nc_invol} we obtain analogous relations for moving $\partial_i^-$ to the right of $\partial_j$. 
The relations can be encoded diagrammatically.  Denote $\partial_i$, respectively $\partial_i^-$, by a crossing of the $i$-th and $(i+1)$-st strands, respectively by a crossing enveloped by a small circle, see Figure~\ref{fig_crossings}. 

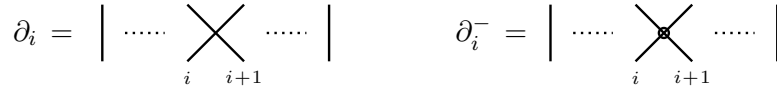
\begin{figure}[!ht]
\[\hackcenter{\begin{tikzpicture}[scale=0.375]
      \node at (-2,1) {$\pd_i \, = \, $};
        \draw (0,0)--(0,2);
        \draw[dotted] (0.8,1)--(2.2,1);
        \draw (3,0)--(5,2);
        \node at (3,-0.6) {\tiny $i$};
        \draw(5,0)--(3,2);
        \node at (5,-0.6) {\tiny $i\!+\!1$};
        \draw[dotted] (5.8,1)--(7.2,1);
        \draw (8,0)--(8,2);
    \end{tikzpicture}}
    \qquad\qquad 
    \hackcenter{\begin{tikzpicture}[scale=0.375]
      \node at (-2,1) {$\pd_i^- \, = \, $};
        \draw (0,0)--(0,2);
        \draw[dotted] (0.8,1)--(2.2,1);
        \draw (3,0)--(5,2);
        \draw (4,1) circle (5pt);
        \node at (3,-0.6) {\tiny $i$};
        \draw(5,0)--(3,2);
        \node at (5,-0.6) {\tiny $i\!+\!1$};
        \draw[dotted] (5.8,1)--(7.2,1);
        \draw (8,0)--(8,2);
    \end{tikzpicture}}
    \]
\caption{Diagrams for the generators $\partial_i,\partial_i^-$.}
\label{fig_crossings}
\end{figure}

Crossings and circled crossings satisfy the defining relations \eqref{eq_demazure} in the nilCoxeter algebra, see Corollary~\ref{cor_relations_minus} and Figure~\ref{nc_relations}. 

\begin{figure}
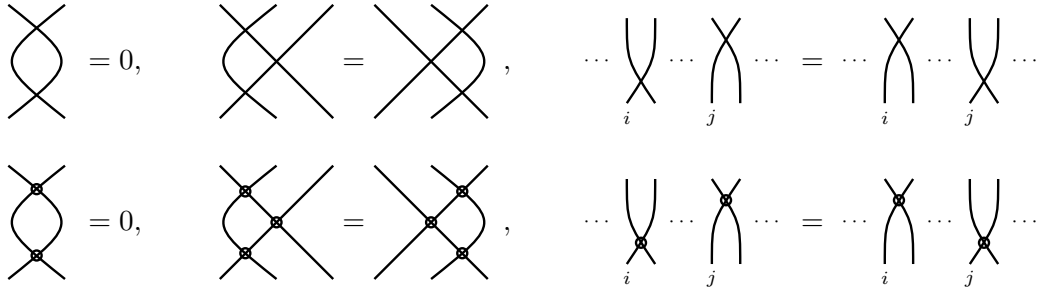

\begin{align*}
    \begin{diagram}
        \draw (0,0) ..controls(2.5,2).. (0,4);
        \draw (2,0)..controls(-0.5,2)..(2,4);
    \end{diagram}
    &=0,\qquad \begin{diagram}
        \draw (0,0)--(4,4);
    \draw (4,0)--(0,4);
    \draw (2,0) ..controls(-0.5,2).. (2,4);
    \end{diagram}
    =
    \begin{diagram}
        \draw (0,0)--(4,4);
    \draw (4,0)--(0,4);
    \draw (2,0) ..controls(4.5,2).. (2,4);
    \end{diagram},\qquad 
    \begin{raisediagram}[-0.5em]
        \node at (0,1.5) {\tiny $\cdots$};
        \draw (3-2,0) ..controls(4-2,1.5).. (4-2,3);
        \draw (4-2,0)..controls(3-2,1.5)..(3-2,3);
        \node at (3,1.5) {\tiny $\cdots$};
        \draw (8-4,3)..controls(9-4,1.5)..(9-4,0);
        \draw (9-4,3)..controls(8-4,1.5)..(8-4,0);
        \node at (6,1.5) {\tiny $\cdots$};
        \node at (1,-0.5) {\tiny $i$};
        \node at (4,-0.5) {\tiny $j$};
    \end{raisediagram}=\begin{raisediagram}[-0.5em]
        \node at (0,1.5) {\tiny $\cdots$};
        \draw (3-2+3,0) ..controls(4-2+3,1.5).. (4-2+3,3);
        \draw (4-2+3,0)..controls(3-2+3,1.5)..(3-2+3,3);
        \node at (3,1.5) {\tiny $\cdots$};
        \draw (8-4-3,3)..controls(9-4-3,1.5)..(9-4-3,0);
        \draw (9-4-3,3)..controls(8-4-3,1.5)..(8-4-3,0);
        \node at (6,1.5) {\tiny $\cdots$};
        \node at (1,-0.5) {\tiny $i$};
        \node at (4,-0.5) {\tiny $j$};
    \end{raisediagram}
\end{align*}

\begin{align*}
    \begin{diagram}
        \draw (0,0) ..controls(2.5,2).. (0,4);
        \draw (2,0)..controls(-0.5,2)..(2,4);
        \draw (1,0.83) circle (5pt);
        \draw(1,4-0.83) circle (5pt);
    \end{diagram}
    &=0,\qquad \begin{diagram}
        \draw (0,0)--(4,4);
    \draw (4,0)--(0,4);
    \draw (2,0) ..controls(-0.5,2).. (2,4);
    \draw (0.9, 0.91) circle (5pt);
    \draw (0.9,4-0.91) circle (5pt);
    \draw (2,2) circle (5pt);
    \end{diagram}
    =
    \begin{diagram}
        \draw (0,0)--(4,4);
    \draw (4,0)--(0,4);
    \draw (2,0) ..controls(4.5,2).. (2,4);
    \draw (2,2) circle (5pt);
    \draw (4-0.9,4-0.91) circle (5pt);
    \draw (4-0.9, 0.91) circle (5pt);
    \end{diagram},\qquad 
    \begin{raisediagram}[-0.5em]
        \node at (0,1.5) {\tiny $\cdots$};
        \draw (3-2,0) ..controls(4-2,1.5).. (4-2,3);
        \draw (4-2,0)..controls(3-2,1.5)..(3-2,3);
        \node at (3,1.5) {\tiny $\cdots$};
        \draw (8-4,3)..controls(9-4,1.5)..(9-4,0);
        \draw (9-4,3)..controls(8-4,1.5)..(8-4,0);
        \node at (6,1.5) {\tiny $\cdots$};
        \draw (1.5,0.75) circle (5pt);
        \draw (4.5, 3-0.75) circle (5pt);
        \node at (1,-0.5) {\tiny $i$};
        \node at (4,-0.5) {\tiny $j$};
    \end{raisediagram}=\begin{raisediagram}[-0.5em]
        \node at (0,1.5) {\tiny $\cdots$};
        \draw (3-2+3,0) ..controls(4-2+3,1.5).. (4-2+3,3);
        \draw (4-2+3,0)..controls(3-2+3,1.5)..(3-2+3,3);
        \node at (3,1.5) {\tiny $\cdots$};
        \draw (8-4-3,3)..controls(9-4-3,1.5)..(9-4-3,0);
        \draw (9-4-3,3)..controls(8-4-3,1.5)..(8-4-3,0);
        \node at (6,1.5) {\tiny $\cdots$};
        \draw (1.5,3-0.75) circle (5pt);
        \draw (4.5, 0.75) circle (5pt);
        \node at (1,-0.5) {\tiny $i$};
        \node at (4,-0.5) {\tiny $j$};
    \end{raisediagram}
\end{align*}
\caption{NilCoxeter algebra relations on crossings and circled crossings.}
\label{nc_relations}
\end{figure}

Relations in Proposition~\ref{prop_relations} are shown in Figure~\ref{fig_more_rels}.  

\begin{figure}[!ht]
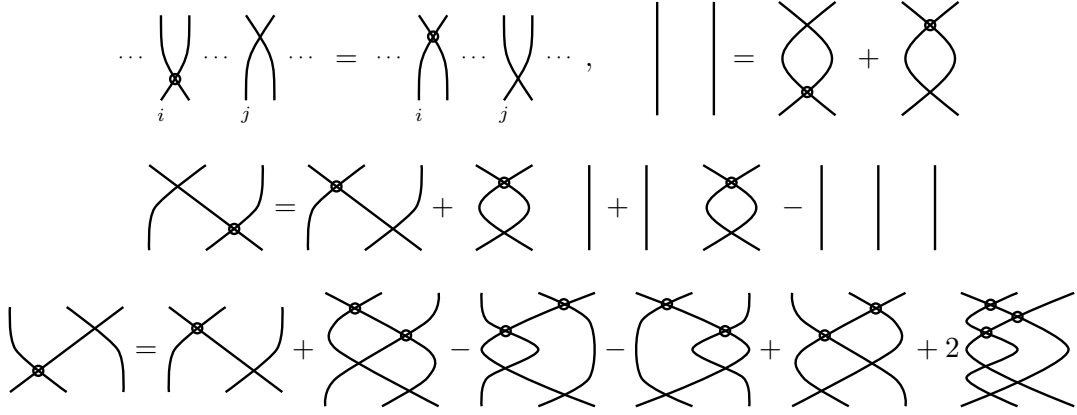
 
\[
\begin{raisediagram}[-0.5em]
        \node at (0,1.5) {\tiny $\cdots$};
        \draw (3-2,0) ..controls(4-2,1.5).. (4-2,3);
        \draw (4-2,0)..controls(3-2,1.5)..(3-2,3);
        \node at (3,1.5) {\tiny $\cdots$};
        \draw (8-4,3)..controls(9-4,1.5)..(9-4,0);
        \draw (9-4,3)..controls(8-4,1.5)..(8-4,0);
        \node at (6,1.5) {\tiny $\cdots$};
        \draw (1.5,0.75) circle (5pt);
        \node at (1,-0.5) {\tiny $i$};
        \node at (4,-0.5) {\tiny $j$};
    \end{raisediagram}=\begin{raisediagram}[-0.5em]
        \node at (0,1.5) {\tiny $\cdots$};
        \draw (3-2+3,0) ..controls(4-2+3,1.5).. (4-2+3,3);
        \draw (4-2+3,0)..controls(3-2+3,1.5)..(3-2+3,3);
        \node at (3,1.5) {\tiny $\cdots$};
        \draw (8-4-3,3)..controls(9-4-3,1.5)..(9-4-3,0);
        \draw (9-4-3,3)..controls(8-4-3,1.5)..(8-4-3,0);
        \node at (6,1.5) {\tiny $\cdots$};
        \draw (1.5,3-0.75) circle (5pt);
        \node at (1,-0.5) {\tiny $i$};
        \node at (4,-0.5) {\tiny $j$};
    \end{raisediagram},\qquad 
\begin{diagram}
    \draw (0,0)--(0,4);
    \draw (2,0)--(2,4);
\end{diagram}
\ =
\begin{diagram}
    \draw (0,0)..controls(2.5,2)..(0,4);
    \draw (2,0)..controls(-0.5,2)..(2,4);
    \draw (1,0.82) circle (5pt);
\end{diagram}
+
\begin{diagram}
    \draw (0,0)..controls(2.5,2)..(0,4);
    \draw (2,0)..controls(-0.5,2)..(2,4);
    \draw (1,3.18) circle (5pt);
\end{diagram}
\]
\vspace{0.5em}
\[
\begin{diagram}
    \draw (0,3)--(4,0);
    \draw plot [smooth,tension=0.5] coordinates { (0,0) (0.2,1.5) (2,3)};
    \draw plot [smooth,tension=0.5] coordinates { (4,3) (3.8,1.5) (2,0)};
    \draw (3,0.75) circle (5pt);
\end{diagram}
=
\begin{diagram}
    \draw (0,3)--(4,0);
    \draw plot [smooth,tension=0.5] coordinates { (0,0) (0.2,1.5) (2,3)};
    \draw plot [smooth,tension=0.5] coordinates { (4,3) (3.8,1.5) (2,0)};
    \draw (1,2.25) circle (5pt);
\end{diagram}
+
\begin{diagram}
    \draw (0,0) ..controls(2.5,1.5).. (0,3);
    \draw (2,0)..controls(-0.5,1.5)..(2,3);
    \draw (4,0)--(4,3);
    \draw (1,2.38) circle (5pt);
\end{diagram}
\ +\ 
\begin{diagram}
    \draw (2,0) ..controls(4.5,1.5).. (2,3);
    \draw (4,0)..controls(1.5,1.5)..(4,3);
    \draw (0,0)--(0,3);
    \draw (3,2.38) circle (5pt);
\end{diagram}
-\ 
\begin{diagram}
    \draw (0,0)--(0,3);
    \draw (2,0)--(2,3);
    \draw (4,0)--(4,3);
\end{diagram}
\]
\vspace{1em}
\[
\begin{diagram}
    \draw (0,0)--(4,3);
    \draw plot [smooth,tension=0.5] coordinates { (4,0) (4-0.2,1.5) (2,3)};
    \draw plot [smooth,tension=0.5] coordinates { (0,3) (4-3.8,1.5) (2,0)};
    \draw (4-3,0.75) circle (5pt);
\end{diagram}
=
\begin{diagram}
    \draw (0,3)--(4,0);
    \draw plot [smooth,tension=0.5] coordinates { (0,0) (0.2,1.5) (2,3)};
    \draw plot [smooth,tension=0.5] coordinates { (4,3) (3.8,1.5) (2,0)};
    \draw (1,2.25) circle (5pt);
\end{diagram}
+
\begin{diagram}
    \draw plot [smooth,tension=0.5] coordinates {(0,0) (0.4,1) (3.6,3) (4,4)};
    \draw plot [smooth,tension=0.5] coordinates {(2,0) (3.6,1) (3.6,2) (0,4)};
    \draw plot [smooth,tension=0.5] coordinates {(4,0) (0.4,2) (0.4,3) (2,4)};
    \draw (1.05,3.45) circle (5pt);
    \draw (2.85,2.5) circle (5pt);
\end{diagram}
-
\begin{diagram}
    \draw plot [smooth,tension=0.5] coordinates {(0,0) (0.2,1) (2,2) (0.2,3) (0,4)};
    \draw plot [smooth,tension=0.5] coordinates {(2,0) (3.8,1) (3.8,3) (2,4)};
    \draw plot [smooth,tension=0.5] coordinates {(4,0) (0,2) (4,4)};
    \draw (0.85,2.65) circle (5pt);
    \draw (2.9,3.6) circle (5pt);
\end{diagram}
-
\begin{diagram}
    \draw plot [smooth,tension=0.5] coordinates {(4-0,0) (4-0.2,1) (4-2,2) (4-0.2,3) (4-0,4)};
    \draw plot [smooth,tension=0.5] coordinates {(4-2,0) (4-3.8,1) (4-3.8,3) (4-2,4)};
    \draw plot [smooth,tension=0.5] coordinates {(4-4,0) (4-0,2) (4-4,4)};
    \draw (4-0.85,2.65) circle (5pt);
    \draw (4-2.9,3.6) circle (5pt);
\end{diagram}
+
\begin{diagram}
    \draw plot [smooth,tension=0.5] coordinates {(4-0,0) (4-0.4,1) (4-3.6,3) (4-4,4)};
    \draw plot [smooth,tension=0.5] coordinates {(4-2,0) (4-3.6,1) (4-3.6,2) (4-0,4)};
    \draw plot [smooth,tension=0.5] coordinates {(4-4,0) (4-0.4,2) (4-0.4,3) (4-2,4)};
    \draw (4-1.05,3.45) circle (5pt);
    \draw (4-2.85,2.5) circle (5pt);
\end{diagram}
+2
\begin{diagram}
    \draw plot [smooth,tension=0.5] coordinates {(0,0) (3.8,2) (0,4)};
    \draw plot [smooth,tension=0.5] coordinates {(2,0) (0.2,1) (2,2) (0.2,3) (2,4)};
    \draw plot [smooth,tension=0.5] coordinates {(4,0) (0.2,2) (4,4)};
    \draw (0.9,2.6) circle (5pt);
    \draw (2,3.15) circle (5pt);
    \draw (1.05,3.58) circle (5pt);
\end{diagram}
\]
\caption{Diagrammatics for Proposition~\ref{prop_relations} relations.}
\label{fig_more_rels}
\end{figure}

\begin{remark}
Relations \eqref{eq_short} and \eqref{eq_long} have more terms than any relation in the standard set of defining relations for the nilHecke algebra. Double nilCoxeter algebra $\wNC_n$ has more complicated relations, in a sense, than $\NH_n$. It is often convenient to do computations in $\wNC_n$ via its embedding in $\NH_n$. Generators  $\NH_n$ are more local than the negative divided difference generator of $\wNC_n$, since $\partial_i^-=-x_i\partial_ix_{i+1}$ is a signed product of three generators of $\NH_n$. 
\end{remark}

\begin{example} Algebras $\wNC_1\cong \kk\cong \wNC_0$. Algebra $\wNC_2$ has
  generators $\partial_1,\partial_1^-$ and defining relations
  \[\partial_1\partial_1=0, \ \ \partial_1^-\partial_1^-=0, \ \
  \partial_1\partial_1^- + \partial_1^-\partial_1 = 1.
  \]
  It is isomorphic to the matrix algebra $\Mat_2(\kk)$, via the map
  \begin{equation}
\partial_1 \mapsto \begin{pmatrix} 0 & 1 \\ 0 & 0 \end{pmatrix}, \
\partial_1^- \mapsto \begin{pmatrix} 0 & 0 \\ 1 & 0 \end{pmatrix}, \
\partial_1\partial_1^- \mapsto
\begin{pmatrix} 1 & 0 \\ 0 & 0 \end{pmatrix}, \
  \partial_1^-\partial_1 \mapsto
  \begin{pmatrix} 0 & 0 \\ 0 & 1 \end{pmatrix}.
  \end{equation}
  Note that elements $\partial_1\partial_1^-$ and $\partial_1^-\partial_1$
  are conjugate minimal idempotents in this algebra.
\end{example}
 
\begin{cor} \label{cor_product_form} Any element of $\wNC_n$ is a $\kk$-linear combination of elements
  of the form $\partial^-_{w_1}\partial_{w_2}$ for $w_1,w_2\in S_n$.
\end{cor}

Either of Corollaries~\ref{cor_product_form} and~\ref{cor_rho_injective} implies
that $\wNC_n$ is a finite-dimensional algebra and $\dim \wNC_n\le (n!)^2$.
The multiplication map 
\begin{equation}\NC^-_n\otimes_{\kk}\NC_n\lra \wNC_n
\end{equation} is surjective (this is a map of $(\NC_n^-,\NC_n)$-bimodules).

Let $w_0=(1,n)(2,n-1)\dots$  denote the longest symmetric group element. 
For each $w\in S_n$, define the Schubert polynomial 
\begin{equation}\label{eq_schubert_2}
\mfS_w \ := \ \partial_{w^{-1}w_0}(\underline{x}_n), \ \ \ \ \underline{x}_n=x_1^{n-1}x_2^{n-2}\dots x_{n-1},
\end{equation}
as well as the \emph{negative Schubert polynomial} 
\begin{equation}\label{eq_neg_schubert}
\mfS^-_w \ := \ \partial^-_{w}(1).
\end{equation}
The Schubert polynomials form a $\kk$-basis of $V_n$, over all $w\in S_n$. Observing that
\[\mfS^-_w=\partial^-_w(1) = \tau \partial_w\tau(1) = \tau \partial_w\underline{x}_n=\tau \mfS_{w_0w^{-1}} \]
and by Proposition~\ref{prop_involution_endomorphism}, we see that the negative Schubert polynomials also form a basis of $V_n$. 

In what follows, we use the left weak order on $S_n$, so that
\[
u\le_L w \quad\Longleftrightarrow\quad
w=tu\ \text{for some }t\in S_n\text{ with }
\ell(w)=\ell(t)+\ell(u).
\]

\begin{prop}\label{prop_rho_iso}
    The map $\rho: \wNC_n \lra \End_{\kk}(V_n)$ in \eqref{eq_hom_wNC} is an isomorphism of $\kk$-algebras.
\end{prop}
\begin{proof}
By Corollary~\ref{cor_rho_injective}, we only need to show that $\rho$ is surjective. We use a triangularity argument with respect to two bases: the Schubert polynomials and the negative Schubert polynomials.
    
We show that the rank 1 linear endomorphism of $V_n$ given by $\mfS_w \mapsto \delta_{u,w}\mfS_v^-$ lies in the image of $\rho$ for every $u,v\in S_n$.
Induct on $\ell(u)$ downwards, the base case being $u=w_0$. In this case, we can take $\partial_v^-\partial_{w_0}$, since 
\[\partial_v^-\partial_{w_0}(\mfS_w) = \delta_{w_0,w}\partial_v^-(1) = \delta_{w_0,w}\mfS_v^- .\]
For the inductive step, consider the action of $\partial_v^-\partial_u$:
\[\partial_v^-\partial_{u}(\mfS_w) = \begin{cases}
    \mfS_v^- &\text{ if }u=w, \\ \partial_v^-(\mfS_t) &\text{ if } w>_L u \text{, so that } w=tu \text { where } l(w)=l(t)+l(u),\\
        0 &\text{ otherwise}.
\end{cases}\]
By induction, for a given $w >_L u$, we may subtract from $\partial^-_v \partial_u$ a linear combination of maps that cancels the term $\partial^-_v (\mfS_t)$. Doing this for each $w>_L u$, we obtain the desired map $\mfS_w \mapsto \delta_{u,w}\mfS_v^-$.
\end{proof}

Proposition~\ref{prop_rho_iso} also follows from~\cite[Theorem 1.9]{EQ23}, and see Proposition~\ref{prop_sub_M} below.  

\begin{prop} \label{prop_wNC_def_rel} The following is a set of defining relations in $\wNC_n$, with generators $\partial_1,\dots, \partial_{n-1}$ and $\partial_1^-,\dots, \partial_{n-1}^-$: 
\begin{eqnarray}
    \partial_i^2 & =& 0, \ \ \partial_i \partial_{i+1}\partial_i = \partial_{i+1}\partial_i\partial_{i+1}, \ \  \partial_i\partial_j = \partial_j \partial_i, \ \ |i-j|>1, \\
    \partial^-_i\partial_i^- & =& 0, \ \ \partial^-_i \partial^-_{i+1}\partial^-_i = \partial^-_{i+1}\partial^-_i\partial^-_{i+1}, \ \  \partial^-_i\partial^-_j = \partial^-_j \partial^-_i, \ \ |i-j|>1, \\
    \partial_i \partial_j^-  & = & \partial_j^- \partial_i \ \ \mathrm{if} \ |i-j|>1, \\
    1 & = & \partial_i \partial_i^- + \partial_i^-\partial_i, \\
    \label{eq_shortt}
    \partial_i \partial_{i+1}^- & = & (\partial_i^- +\partial_{i+1}^-)\partial_{i+1}  + \partial_i^- \partial_i -1
    , \\
    \label{eq_longg}
    \partial_{i+1}\partial_i^-  & = & \partial_i^-\partial_{i+1} + 
    [\partial_i^-,\partial_{i+1}^-]\,  [\partial_i,\partial_{i+1}]+2\partial_{i,i+1,i}^-\partial_{i,i+1,i}
\end{eqnarray}
\end{prop} 
Note that relation \eqref{eq_shortt} can be rewritten as 
\[\partial_i (\partial_{i+1}^- +   \partial_i^-) = (\partial_i^- +\partial_{i+1}^-)\partial_{i+1}\]
\begin{proof}
    Denote by $\wNC_n'$ the algebra with these generators and defining relations. These relations are obtained as the union of the relations in \eqref{eq_demazure}, analogous relations on $\partial^-_i$'s, and the relations in Proposition~\ref{prop_relations}. All of these relations hold in $\wNC_n$, and there is a surjective homomorphism $\wNC_n'\lra \wNC_n$.

    Let us show that the relations allow us to write any element of $\wh\NC_n'$ in the span of $\pd_w^-\pd_u$. Informally, the argument below works since we're acting on a finite-dimensional space and the algebra is nilpotent in both positive and negative directions. To do so, it suffices show that
    \[
    \pd_w^-\pd_u\cdot\pd_i^-\in \Span_\bk\{\pd_x^-\pd_y:w\le_\te{R}x,\ \ell(x)-\ell(y)=\ell(w)-\ell(u)+1\},
    \] 
    where $w\le_\te{R}x$ means $wv=x$ and $\ell(w)+\ell(v)=\ell(x)$ (note that then the induction hypothesis says $\ell(v)+\ell(u)-1=\ell(y)$). 
    We prove this by induction on the lexicographic ordering of $\pr*{\binom{n}{2}-\ell(w),\ell(u)}$. 
    For the base case, if $u=\id$ and $w$ is anything, the conclusion is clear. 
    For the induction step, let $u=u's_j$ ($\ell(u)>\ell(u')$), so that $\pd_w^-\pd_u\pd_i^-=\pd_w^-\pd_{u'}\pd_j\pd_i^-$. By the relations, we know
    \[
    \pd_j\pd_i^-=\sum_{\substack{x',y'\\ \ell(x')=\ell(y')\in[0,3]}} c_{x',y'} \pd_{x'}^-\pd_{y'}.
    \] 
    Then write
    \[
    \pd_w^-\pd_u\cdot\pd_i^-=\sum_{\substack{x',y'\\ \ell(x')=\ell(y')\in[0,3]}} c_{x',y'} \pd_w^-\pd_{u'} \pd_{x'}^-\pd_{y'}.
    \] 
    For each term, let $x'=s_{i_1}\cdots s_{i_k}$, and let us try to reorder $\pd_w^-\pd_{u'}\pd_{i_1}^-\cdots\pd_{i_k}^-$. 
    The first step of reordering $\pd_w^-\pd_{u'}\pd_{i_1}^-$ is possible by hypothesis and lands in the span of $\pd_x^-\pd_y$; at each next step either $x$ has strictly increased, and therefore the first induction coordinate has decreased, or $x$ has not changed and $y$ has decreased in length. 
    Hence we may reorder $\pd_w^-\pd_{u'} \pd_{x'}^-$ by induction into the span of $\pd_x^-\pd_y$ such that $w\le_\te{R} x,\ \ell(x)-\ell(y)=\ell(w)-\ell(u')+\ell(x')$. Now the sum looks like 
    \[
    \pd_w^-\pd_u\cdot\pd_i^-=\sum_{\substack{x',y'\\ \ell(x')=\ell(y')\in[0,3]}} c_{x',y'} \sum_{\substack{x,y}}c_{x,y}\pd_x^-\pd_{y}\pd_{y'}.
    \]
    Note that either $\pd_y\pd_{y'}=0$ or $\ell(yy')=\ell(y)+\ell(y')$ and $w\le_\te{R} x$ and $\ell(x)-\ell(yy')=\ell(w)-\ell(u')+\ell(x')-\ell(y')=\ell(w)-\ell(u)+1$, which completes the induction step. 

    Thus, the relations imply that elements $\partial_{w_1}^-\partial_{w_2}$ over all $w_1,w_2\in S_n$ span $\wNC_n'$, so that $\dim(\wNC_n')\le (n!)^2$.   Since $\dim(\wNC_n)=(n!)^2$, the Proposition follows.
\end{proof}

\begin{remark} Inclusion homomorphism $\wNC_n\lra \NH_n$ is given diagrammatically by taking the crossing $\partial_i$ in $\wNC_n$ to the corresponding divided difference operator $\partial_i$ in $\NH_n$ and circled crossing $\partial_i^-$ to the diagram on the RHS of \eqref{eq_circle_crossing}.   
\end{remark}

\begin{example} \label{ex_problem} Relations
  \begin{eqnarray}
  \partial^-_i \partial_{i+1}\partial_i & = & \partial_{i+1}\partial_i \partial^-_{i+1}, \ \ \
  \partial_i \partial^-_{i+1}\partial_i^- \ = \ \partial^-_{i+1}\partial^-_i \partial_{i+1},
  \\ 
   {} [\partial_i,\partial_{i+1}]\partial_i^- & = & \partial_{i+1}^-[\partial_i,\partial_{i+1}], 
  \  \ \ 
  [\partial^-_i,\partial^-_{i+1}]\partial_i \ =\  \partial_{i+1}[\partial_i^-,\partial_{i+1}^-]
  \end{eqnarray} 
  hold in $\wNC_n$. The two equalities in each row follow from each other by applying the involution $\tau$ of $\wNC_n$. Either one can be checked directly in $\NH_n$, which contains $\wNC_n$ as a subalgebra. 
\end{example}

Consider a particular matrix subalgebra $\mathsf{M}_n$ of $\NH_n$ of size $n!$ described in \cite[Section 2.5]{KLMS12}. 
 Denote by $\mathrm{Seq}(n)$ the set of sequences $\{\alpha=(\alpha_1,\dots, \alpha_{n-1})|~0\leq \alpha_t \leq
t,~t=1,\cdots, n-1\}$  For any
$\alpha\in \mathrm{Seq}(n)$, let $|\alpha|=\sum_{t=1}^{n-1}\alpha_t$ and
$\hat{\alpha}=(\hat{\alpha}_1,\hat{\alpha}_2, \dots, \hat{\alpha}_{n-1}
):=(1-\alpha_1,2-\alpha_2,\dots,n-1-\alpha_{n-1})$. For
 $\alpha \in \mathrm{Seq}(n)$ define
$$e_{\alpha}:=e_{\alpha_1}^{(1)}e_{\alpha_2}^{(2)}\cdots
e_{\alpha_{n-1}}^{(n-1)},\ \ \ \
x^{\hat{\alpha}}:=x_2^{\hat{\alpha}_1}x_3^{\hat{\alpha}_2}\cdots
x_{n}^{\hat{\alpha}_{n-1}},$$
where
\begin{equation}\label{eq_elem}
e_m^{(k)} = \sum_{1\le i_1 < \dots < i_m \le k} x_{i_1}\cdots x_{i_m}
\end{equation}
is the $m$-th elementary symmetric function in variables $x_1,\dots, x_k$. 

Denote by $\mathsf{M}_n$ the $\kk$-span of elements 
\[\{E_{\alpha,\beta}=(-1)^{|\hat{\beta}|}
e_{\alpha}\,\partial_{w_0}\,x^{\hat{\beta}}
~|~\alpha,~\beta \in \mathrm{Seq}(n)\},\]
where 
\begin{equation}\label{eq_wnull}
 \partial_{w_0} = (\partial_1\dots \partial_{n-1})(\partial_1\dots \partial_{n-2})\dots (\partial_1) 
\end{equation}
is the iterated divided difference operator for the longest symmetric group element $w_0$.

\begin{prop} \cite[Section 2.5]{KLMS12} \label{prop_basis_NH} 
Elements $E_{\alpha,\beta}$ satisfy the relations 
\[
E_{\alpha,\beta} E_{\beta',\gamma}=\delta_{\beta,\beta'}E_{\alpha,\gamma}
\]
and constitute a homogeneous basis of a matrix subalgebra of $\NH_n$ of size $n!$, so that 
\[
\Mat_{n!}(\kk)\cong \mathsf{M}_n  \subset \NH_n.
\]
The inclusion $\mathsf{M}_n\subset \NH_n$ extends to an isomorphism 
\[
\mathsf{M}_n\otimes\Sym_n\cong \NH_n.
\]
\end{prop}

The subalgebra also appears in \cite[Proposition~3.3]{KQ15}. See~\cite{KLMS12,KQ15} for a diagrammatic description of this subalgebra and a proof of the above result. 

\begin{remark}
    Other subalgebras of $\NH_n$ isomorphic to $\Mat_{n!}(\kk)$ can be obtained by conjugating $\mathsf{M}_n$ by a matrix $A$ in $\mathsf{GL}(n!,\Sym_n)$. To get a graded subalgebra, entries of $A$ must be homogeneous of suitable degrees. 
\end{remark}

\begin{remark}
    $\wNC_n$ is a maximal finite-dimensional subalgebra of $\NH_n$. It is straightforward to check that any strictly larger subalgebra $A$ with $\wNC_n\subset A \subset \NH_n$ is infinite-dimensional.
\end{remark}

For the discussion below of $\slt$ actions, assume that $\kk$ is a field of characteristic $0$. 
B.~Elias and Y.~Qi~\cite[Section 1.3]{EQ23} define an action of the Lie algebra $\slt$ on the nilHecke algebra $\NH_n$,  with the basis elements $\mathsf{e},\mathsf{f},\mathsf{h}\in \slt $ acting (by derivations) as follows: 
\begin{eqnarray}\label{eq_action_1}
     & & \mathsf{e}(x_i)=-x_i^2, \ \mathsf{e}(\partial_i)=x_i\partial_i+\partial_i x_{i+1}, \  \\
     \label{eq_action_2}
     & & \mathsf{f}(x_i) = 1, \ \mathsf{f}(\partial_i)=0,  \\
     \label{eq_action_3}
     & & \mathsf{h}(u) = \deg(u)u, \ \ u\in \NH_n, \ \deg(x_i)=2, \deg(\partial_i)=-2, \\
     & & \mathsf{t}(uv)=\mathsf{t}(u)v + u \mathsf{t}(v), \ \mathsf{t}\in\slt,\ u,v\in \NH_n,
\end{eqnarray}
Viewing $\NH_n$ as a representation of $\slt$, take its maximal finite-dimensional subrepresentation and denote it $\NH_n^{\mathsf{fd}}$. Since $\mathsf{e},\mathsf{f}$ act by derivations,  subrepresentation $\NH_n^{\mathsf{fd}}$ is a subalgebra. 

\begin{prop} \cite[Theorem~1.9]{EQ23} \label{prop_sub_M} Subalgebras $\NH_n^{\mathsf{fd}}$ and $\mathsf{M}_n$ of $\NH_n$ coincide:
\[
\NH_n^{\mathsf{fd}} \cong \mathsf{M}_n.
\]
\end{prop}

\begin{prop} Subalgebras $\wNC_n$ and $\NH_n^{\mathsf{fd}}$ of $\NH_n$ coincide. 
\end{prop}
\begin{proof}
Elements $\partial_i$ and $\partial_i^-$ are in $\NH_n^{\mathsf{fd}}$. It is enough to check that for $i=1$ and $n=2$, which is an easy computation. Algebra $\NH_2^{\mathsf{fd}}$ is written down explicitly in~\cite[Example 1.10]{EQ23}. It is a $2\times 2$ matrix  algebra with a basis $\{\partial_1\partial_1^-,\partial_1^-\partial_1,\partial_1,\partial_1^-\}$, thus it contains $\partial_1$ and $\partial_1^-$. 
Note also that $\partial_1,\partial_1^-$ and 
$\mathsf{e}(\partial_1)=2x_1\partial_1-1=\partial_1^-\partial_1-\partial_1\partial_1^-$  span a copy of the adjoint representation of $\slt$ in $\NH_2$ and $\wNC_2$.

    Therefore, algebra $\NH_n^{\mathsf{fd}}$ contains generators of $\wNC_n$ and there is an inclusion $\wNC_n\subset \NH_n^{\mathsf{fd}}$. Since the two algebras have the same dimension, the inclusion is an equality.
\end{proof}

 \begin{cor}\label{cor_three_coincide}
     Subalgebras $\wNC_n$, $\mathsf{M}_n$, and $\NH_n^{\mathsf{fd}}$ of $\NH_n$ coincide:
     \[
     \wNC_n = \mathsf{M}_n =\NH_n^{\mathsf{fd}} .
     \]
 \end{cor}

\begin{remark}
Define $\ol\NH_n$ as the subalgebra of endomorphisms of $R_n^-$ (or of $R_n^{\circ}$) generated by the operators $\opartial_i$, $1\le i \le n-1$ and operators of multiplication by $x_1^{-1},\dots, x_n^{-1}$. Rewriting \eqref{eq_opart_t} as $\partial_i = - x_i^{-1}x_{i+1}^{-1}\opartial_i$ shows that $\partial_i\in \ol\NH_n$ but $\partial_i^-$ is not in $\ol\NH_n$. 
\end{remark}

Recall that $\underline{x}_n$ is given by \eqref{eq_monomial}. 
Define $\NH_n^-$ as the subalgebra of endomorphisms of $R_n^-\underline{x}_n$ (or of $R_n^{\circ}$, which contains $R_n^-\underline{x}_n$) generated by operators $\partial_i^-$, $1\le i \le n-1$ and  operators of multiplication by $x_1^{-1},\dots, x_n^{-1}$. Operators $\partial_i^-$ and multiplication by $x_j^{-1}$ are  conjugate, via involution $\tau$, to the operators $\partial_i$ and multiplication by $x_j$, correspondingly.

Involution $\tau$ of $\kk$-vector space $R_n^{\circ}$ taking $R_n$ to $R_n^-\underline{x}_n$ induces an isomorphism 
\begin{equation}\label{eq_aut_tau_1}
\tau \ : \ \NH_n \stackrel{\cong}{\lra} \NH_n^-, \ \ \tau(\partial_i)=\partial_i^-, \ \ \tau(x_i)=x_i^{-1}  
\end{equation}
(same notation). 

\vspace{0.1in} 

Let $\NH_n^{\circ}$ be the subalgebra of endomorphisms of $R_n^{\circ}$ generated by $\partial_1,\dots, \partial_{n-1}$ and operators of multiplication by $x_1,x_1^{-1},\dots, x_n,x_n^{-1}$. This subalgebra contains $\NH_n, \NH_n^-$ and $\ol\NH_n$ as subalgebras. Note the natural isomorphism 
\begin{equation}
    \NH_n^{\circ} \cong \Mat_{n!}(\kk)\otimes \Sym_n^{\circ}
\end{equation}
induced by the corresponding isomorphism for $\NH_n$ or, equivalently, by the isomorphism \eqref{eq_iso_circc}.

\begin{remark}
Automorphism $\tau$ in \eqref{eq_aut_tau_1} extends to an involution of the algebra $\NH_n^{\circ}$, also denoted $\tau$ and given by 
\begin{equation}\label{eq_aut_tau_2}
\tau(\partial_i)=\partial_i^-, \ \ \tau(\partial_i^-)=\partial_i,  \ \ \tau(x_i^{\pm 1})=x_i^{\mp 1}  
\end{equation}
Consequently, $\tau$ restricts to an involution on $\wNC_n$, also denoted $\tau$ and given by 
\begin{equation}\label{eq_aut_tau_3}
\tau(\partial_i)=\partial_i^-, \ \ \tau(\partial_i^-)=\partial_i,  \ \ 1\le i \le n-1.  
\end{equation}
This involution transposes the relations in the first two rows of Proposition~\ref{prop_wNC_def_rel}. It preserves the relations in the third and fourth rows. Applying it to the relations in the last two rows nets the following additional relations: 
\begin{eqnarray}
    \label{eq_shorttt}
    \partial_i^- \partial_{i+1} & = & (\partial_i +\partial_{i+1})\partial_{i+1}^-  + \partial_i \partial_i^- -1
    , \\
    \label{eq_longgg}
    \partial_{i+1}^-\partial_i  & = & \partial_i\partial_{i+1}^- + 
    [\partial_i,\partial_{i+1}]\,  [\partial_i^-,\partial_{i+1}^-]+2\partial_{i,i+1,i}\partial_{i,i+1,i}^-.
\end{eqnarray}
\end{remark}

\begin{prop}\label{prop_below}
    Elements $\partial_i,\partial_i^-$, for $1\le i\le n-1$, belong to the intersection $\NH_n\cap \NH_n^-$. 
\end{prop}
\begin{proof} We already know that these elements belong to $\NH_n$. By definition, $\partial_i^-\in \NH_n^-$. 
Rewriting equation \eqref{eq_d_minus} as $\partial_i=-x_i^{-1}\partial_i^-x_{i+1}^{-1}$ implies that $\partial_i\in \NH_n^-$. 
\end{proof} 

Lie algebra $\slt$ action on $\NH_n$ given by \eqref{eq_action_1}-\eqref{eq_action_3} extends to its localization $\NH_n^{\circ}$ by 
\begin{equation}\label{eq_action_4}
    \mathsf{e}(x_i^{-1}) =  1, \ \ \mathsf{f}(x_i^{-1}) = - x_i^{-2},\ \ \mathsf{h}(x_i^{-1})=-2 x_i^{-1}. 
\end{equation}
We then compute 
\begin{eqnarray*}& & \mathsf{e}(\partial_i^-) = -\mathsf{e}(x_i\partial_ix_{i+1})=
x_i^2\partial_ix_{i+1}-x_i(x_i\partial_i+\partial_i x_{i+1})x_{i+1}+x_i\partial_i x_{i+1}^2=0. \\
& &  \mathsf{f} (\partial_i^-) =
- \mathsf{f}(x_i\partial_ix_{i+1})= \partial_i^-x_{i+1}^{-1} + x_i^{-1}\partial_i^{-},
\end{eqnarray*}
 and see that the $\slt$ action preserves the subalgebra $\NH_n^-\subset \NH_n^{\circ}$ as well:
\begin{align}
  \mathsf{e}(\pd_i) &= x_i\partial_i+\partial_i x_{i+1} = \pd_i^-\pd_i-\pd_i \pd_i^-,  &  
  \mathsf{h}(\pd_i) &=  -2 \pd_i , & 
  \mathsf{f}(\pd_i) &=  0   ,\\
\mathsf{e}(\pd_i^-) &= 0,    &   
\mathsf{h}(\pd_i^-) &= 2 \pd_i^- , & 
\mathsf{f}(\pd_i^-) &= \pd_i \pd_i^- - \pd_i^-\pd_i. 
\end{align}

\begin{prop}\label{prop_sub_two_equal}
    Subalgebras $\wNC_n$ and $\NH_n\cap\NH_n^-$ of $\NH^{\circ}_n$ are equal. 
\end{prop}

\begin{proof} Proposition~\ref{prop_below} gives the inclusion $\wNC_n\subset \NH_n\cap\NH_n^-$. Graded subalgebras $\NH_n$ and $\NH_n^-$ are bounded from below and from above, respectively. Their intersection lives in finitely many degrees. Since $\slt$ action preserves each of these subalgebras, $\slt$ acts on the intersection as well, and $\NH_n\cap\NH_n^-\subset \NH_n^{\mathsf{fd}}$. Corollary~\ref{cor_three_coincide} and the chain of inclusions $\wNC_n\subset \NH_n\cap\NH_n^-\subset \NH_n^{\mathsf{fd}}$ implies the equality. 

For an alternative proof of the proposition, one can use canonical isomorphisms 
\[
\NH_n\cong \End_{\kk}(V_n)\otimes \Sym_n,  \ \ \NH_n^-\cong \End_{\kk}(V_n)\otimes \Sym_n^-,
\]
and observe that the intersection $\Sym_n\cap \Sym_n^-=\kk$. Here $\Sym_n^-= \kk[x_1^{-1},\dots, x_n^{-1}]^{S_n}\cong Z(\NH_n^-)$ is the subring of symmetric polynomials in $x_1^{-1},\dots, x_n^{-1}$, naturally isomorphic to the center of $\NH_n^-$. 
\end{proof}

\begin{cor}\label{cor_four_equal}
The following subalgebras of $\NH^{\circ}_n$ are equal: 
\begin{equation}
    \wNC_n = \mathsf{M}_n =\NH_n^{\mathsf{fd}} = \NH_n\cap \NH_n^-. 
\end{equation}
\end{cor}

\begin{remark}\label{rmk_tau_slt}
Under the isomorphism $\NH_n\lra \NH_n^-$ induced by $\tau$ and taking $\partial_i$ to $\partial_i^-$ and $x_j$ to $x_j^{-1}$ the actions of $\slt$ on $\NH_n$ and $\NH_n^-$ are intertwined via the Cartan involution 
$\mathsf{e}\mapsto \mathsf{f}, \mathsf{f}\mapsto \mathsf{e}, \mathsf{h}\to -\mathsf{h}$ of $\slt$. 
\end{remark}

\begin{remark}\label{rmk_inv_tauone}
Algebra $\NH_n$ admits an anti-involution $\tau_1$ given by 
\begin{equation}\label{eq_tauone}
    \tau_1(x_i)=x_{n+1-i}, \ \ \tau_1(\partial_j)=-\partial_{n-j}, \ \ 1\le i\le n, 1\le  j \le n-1, 
\end{equation}
since the defining relations \eqref{eq_nilHecke_def} are invariant under this substitution combined with reversing the order of the product of generators ($\tau_1$ rotates a diagram representing a product element of $\NH_n$ by $180^{\circ}$ and scales it by $-1$ if the diagram has odd number of crossings). Then $\tau_1(\partial_j^-)=\tau_1(-x_j\partial_jx_{j+1}) = - x_{n-j}(-\partial_{n-j})x_{n+1-j}= -\partial_{n-j}^-$. 

Consequently, $\tau_1$ restricts to an anti-involution of $\wNC_n$, also denoted $\tau_1$ and given by 
\begin{equation}\label{eq_tau_one}
     \tau_1(\partial_j)=-\partial_{n-j}, \ \ \tau_1(\partial_j^-)=-\partial_{n-j}^-, \ \  1\le  j \le n-1. 
\end{equation}
Applying this antiinvolution to the relations \eqref{eq_shortt} and \eqref{eq_longg} we again get the relations \eqref{eq_shorttt} and \eqref{eq_longgg}, respectively. Anti-involution $\tau_1$ of $\wNC_n$ commutes with the involution $\tau$. 

Likewise, $\tau_1$ extends to an anti-involution on $\NH_n^-$ via $\tau_1(x_i^{-1})=x_{n+1-i}^{-1}$ and to an anti-involution on $\NH_n^{\circ}$. 
\end{remark}

\vspace{0.1in}


\subsection{A minimal categorification of the divided powers ring}\label{subsection_minimal}

Natural inclusions of algebras $\NH_n\otimes \NH_m\lra \NH_{n+m}$ induce
natural inclusions $\NC_n^{\varepsilon}\otimes \NC_m^{\varepsilon}\lra \NC_{n+m}^{\varepsilon}$ for $\varepsilon \in \{+,-\}$. In turn, we
obtain natural inclusions
\begin{equation}\label{eq_nat_inclusions}
  \wNC_n\otimes \wNC_m\lra \wNC_{n+m}.
\end{equation}
This inclusion takes $\partial_i^{\varepsilon}\otimes 1$ to $\partial_i^{\varepsilon}$ and $1\otimes\partial_j^{\varepsilon} $
to $\partial_{n+j}^{\varepsilon}$ for $\varepsilon\in\{+,-\}$.

Consider a $\kk$-linear monoidal category $\mcNC'$ with a generating object $\mcE_0$ and
\begin{equation}\label{eq_cat_C}
  \Hom_{\mcNC'}(\mcE_0^{\otimes n},\mcE_0^{\otimes m}) = \begin{cases} \wNC_n & \mathrm{if} \ n=m, \\ 0 & \mathrm{otherwise} \end{cases}.
\end{equation}
The tensor product is given on morphisms via  maps in \eqref{eq_nat_inclusions}.

\begin{figure}[!h]
\label{fig_crossings_2}
\[ \hackcenter{\begin{tikzpicture}[scale=0.375]
      \node at (-2,1) {$\pd \, = \, $};
        \draw (0,0)--(2,2);
        \draw(2,0)--(0,2);
    \end{tikzpicture}}
    \qquad\qquad 
    \hackcenter{\begin{tikzpicture}[scale=0.375]
      \node at (-2,1) {$\pd^- \, = \, $};
        \draw (0,0)--(2,2);
        \draw (1,1) circle (5pt);
        \draw(2,0)--(0,2); 
    \end{tikzpicture}}
\]
\caption{Diagrams for the generators $\partial,\partial^-$ of $\mcNC'$.}
\end{figure}

Category $\mcNC'$ has two generating morphisms $\partial,\partial^-\in \End_{\mcNC'}(\mcE_0^{\otimes 2})$, represented by crossings of two types and shown in Figure~\ref{fig_crossings_2} with the following defining relations: 

\begin{itemize}
    \item Relations shown in Figure~\ref{nc_relations} on the left: $\partial^2=(\partial^-)^2=0$ and in the center: 
    \begin{eqnarray*}
    (\partial\otimes 1)(1\otimes \partial)(\partial\otimes 1) & = & (1\otimes \partial)(\partial\otimes 1)(1\otimes \partial), \\
 (\partial^-\otimes 1)(1\otimes \partial^-)(\partial^-\otimes 1) & = & (1\otimes \partial^-)(\partial^-\otimes 1)(1\otimes \partial^-),
    \end{eqnarray*}
    \item Relation $\partial \partial^-+\partial^-\partial=1$, shown in Figure~\ref{fig_more_rels}  top right, 
    \item Relations shown in Figure~\ref{fig_more_rels} on lines 2 and 3, obtained from \eqref{eq_shortt} and \eqref{eq_longg} by replacing $\partial_i$, $\partial_{i+1}$, $\partial_i^-$,$\partial_{i+1}^-$ by $\partial\otimes 1$, $1\otimes \partial$, $\partial^-\otimes 1$, $1\otimes \partial^-$, respectively. 
\end{itemize}

\vspace{0.1in} 

Form the Karoubi envelope 
\[\mcNC :=\Kar(\mcNC').
\]
Due to an isomorphism with the matrix algebra: 
\[
\End_{\mcNC'}(\mcE_0^{\otimes n}) \cong \wNC_n\cong \Mat_{n!}(\kk),
\]
the Karoubi envelope contains objects, denoted $\mcE_0^{(n)}$, such that $n! \mcE_0^{(n)}\cong \mcE_0^{\otimes n}$. 
Specifically, in \eqref{eq_id_e2}, idempotent $e_2=\partial_1^-\partial_1 \in \wNC_2$ and $e_{2,i}\in \wNC_n$. Consequently, $e_n\in \wNC_n$ and we can define $\mcE_0^{(n)}:=(\mcE_0^{\otimes n},e_n)$. 
Note that $\End_{\mcNC}(\mcE_0^{(n)})\cong \kk$.

Inclusions of rings $\wNC_n\subset \NH_n$ induce a faithful monoidal functor $\mcNC' \lra \mcNH'$ taking $\mcE_0$ to $\mcE$ and inducing a monoidal functor on their Karoubi envelopes $\mcNC \lra \mcNH$. The latter functor induces an isomorphism of Grothendieck rings of additive monoidal categories 
\[
K_0(\mcNC) \cong K_0(\mcNH) \cong \Z\{E\},  
\]
assuming $K_0(\kk)\cong \Z$, which is the case when $\kk$ is a field, for instance. 

\begin{remark}\label{rmk_graded}
The Lie algebra $\slt$ acts on monoidal categories $\mcNH'$ and  $\mcNC'$, via its action on the endomorphism rings of objects $\mcE^{\otimes n}$ and $\mcE_0^{\otimes n}$. Forming Karoubi completions while keeping track of an $\slt$ action requires care, see~\cite{EQ23}, and we do not consider it here.
\end{remark}

\begin{remark}
    Categories $\mcNH'$ and $\mcNH$, see \eqref{eq_NH_cat_homs}, are  naturally symmetric monoidal, with the transposition morphism of $\mcE^{\otimes 2}$ induced by the transposition automorphism $x_1\leftrightarrow x_2$ on $\kk[x_1,x_2]$. This is an endomorphism of $\kk[x_1,x_2]$ as a $\Sym_2$-module. It can be written as $1+(x_2-x_1) \partial_1 \in \NH_2$. This element is not in $\wNC_2$, and this symmetric monoidal structure does not restrict to one on $\mcNC'$, see \eqref{eq_cat_C}. 
\end{remark}

\begin{remark} Rings $\wNC_n$ are naturally graded, and the minimal categorification above naturally refines to a graded one, giving an example of a minimal categorification of the $q$-deformation $\Z_q\{E\}$ of $\Z\{E\}$, see the earlier discussion following \eqref{eq_q_deform} and Definition~\ref{def_graded} in the next section.  
\end{remark}

%
%

\section{A classification of some minimal categorifications}\label{sec_classify}

We continue to work over a ground field $\kk$. 

\begin{definition}\label{def_min_one}
A \emph{minimal categorification} of the divided powers ring $\Z\{E\}$ is a
$\kk$-linear additive monoidal category $\mcC'$ with a generating object $X$
such that every object is a finite direct sum of objects $X^{\otimes n}$ and
\[
\Hom_{\mcC'}(X^{\otimes n},X^{\otimes m})\cong
\begin{cases}
\Mat_{n!}(\kk),&n=m,\\
0,&n\neq m.
\end{cases}
\]
\end{definition}

Such a category $\mcC'$ may also be called a minimal categorification of the
positive half of $\slt$. Its Karoubi envelope
$\mcC:=\Kar(\mcC')$ contains objects representing the divided powers on the
Grothendieck group. If $K_0(\kk)\cong\Z$, then
\[
K_0(\mcC)\cong \Z\{E\}.
\]

Write
\[
A_n:=\End_{\mcC'}(X^{\otimes n}),
\]
so that $A_0=\bk$. 
The monoidal structure gives unital algebra homomorphisms
\begin{equation}\label{eq_nm_homs}
\psi_{n,m}:A_n\otimes A_m\longrightarrow A_{n+m},\qquad n,m\geq 0,
\end{equation}
which satisfy associativity relations
\begin{equation}\label{eq_assoc_relations}
\psi_{n+m,r}\circ(\psi_{n,m}\otimes \operatorname{id}_{A_r})
=
\psi_{n,m+r}\circ(\operatorname{id}_{A_n}\otimes\psi_{m,r})
\end{equation}
and the unital relations
\begin{equation}\label{eq:psi 0 n}
    \psi_{0,n}=\psi_{n,0}=\id_{A_n}.
\end{equation}
Conversely, a system of matrix algebras $A_n$ (with $A_0=\bk$) and homomorphisms
\eqref{eq_nm_homs} satisfying \eqref{eq_assoc_relations} and \eqref{eq:psi 0 n} determines such a
minimal categorification.

We will use the graded version of this notion.

\begin{definition}\label{def_graded}
 A \emph{graded minimal
categorification} of the quantum divided powers ring $\Z_q\{E\}$ is a
$\kk$-linear $\mathbb Z$-graded additive monoidal category $\mcC'$ with a
generating object $X$ such that every object is a finite direct sum of grading
shifts of the objects $X^{\otimes n}$ and
\[
\mathsf{HOM}_{\mcC'}(X^{\otimes n},X^{\otimes m})\cong
\begin{cases}
\Mat_{[n]!}(\kk),&n=m,\\
0,&n\neq m,
\end{cases}
\]
as graded algebras when $n=m$. Here
\[
\mathsf{HOM}_{\mcC'}(M,N)
:=\bigoplus_{j\in\mathbb Z}\Hom_{\mcC'}(M\{j\},N)
\]
and $\Mat_{f(q)}(\kk)$ denotes the graded endomorphism algebra of a free graded
module of graded dimension $f(q)$, for $f(q)\in\BN[q,q^{-1}]$ a Laurent polynomial.
\end{definition}

In the graded setting we use the same notation
\[
A_n:=\mathsf{HOM}_{\mcC'}(X^{\otimes n},X^{\otimes n})
\]
for the full graded endomorphism algebra. The maps $\psi_{n,m}$ should be homomorphisms of $\Z$-graded algebras. 
 Then the
natural map $\Z_q\{E\}\to K_0(\Kar(\mcC'))$ is an isomorphism.

For the rest of this section
we work with graded minimal categorifications. Equivalently, we work
with the systems $(A_n,\psi_{n,m})$ above. 
We classify these categories up to  graded monoidal
isomorphism preserving the chosen generator.
 Such an isomorphism
induces graded algebra isomorphisms
\[
\Phi_n:A_n\stackrel{\sim}{\longrightarrow}A_n'
\]
commuting with all maps $\psi_{n,m}$.

For each $n$, choose the unique overall grading shift of the simple graded
$A_n$-module $W_n$ for which
\[
A_n\cong\End_{\kk}(W_n),\qquad \gdim W_n=[n]!.
\]
Thus the highest and lowest degrees of $W_n$ are respectively
$\binom{n}{2}$ and $-\binom{n}{2}$.

\subsection{The two- and three-strand structure}\label{subsec:abstract 2,3 string structure}

We have $A_0\cong A_1\cong\kk$ and
$A_2\cong\End_{\kk}(W_2)$, where
\[
W_2\cong\kk\{-1\}\oplus\kk\{1\}.
\]
Choose homogeneous elements
\[
\partial\in A_2^{-2},\qquad \partial^-\in A_2^2
\]
satisfying
\begin{equation}\label{eq_same_i}
\partial^2=(\partial^-)^2=0,
\qquad
\partial\partial^-+\partial^-\partial=1.
\end{equation}
The pair is unique up to scaling 
\begin{equation}\label{eq_rescale_pair}
\partial\longmapsto t\partial,
\qquad
\partial^-\longmapsto t^{-1}\partial^-,
\qquad t\in\kk^\times.
\end{equation}
For $1\leq i\leq n-1$, let $\partial_i,\partial_i^-\in A_n$ be the images of
these elements under the inclusion which places $A_2$ in positions $i,i+1$.
Then
\begin{equation}\label{eq_same_i_3}
\partial_i^2=(\partial_i^-)^2=0,
\qquad
\partial_i\partial_i^-+\partial_i^-\partial_i=1.
\end{equation}
Let $A_n^+$ and $A_n^-$ be the subalgebras generated by the $\partial_i$ and
$\partial_i^-$, respectively.

We first analyze three strands. Let $v$ span the top degree $W_3^3$ and assume
\begin{equation}\label{eq_positive_cyclic_three}
A_3^+v=W_3.
\end{equation}
Then
\begin{equation}\label{eq_basis_W3}
v,\ \partial_1v,\ \partial_2v,\ \partial_{12}v,\
\partial_{21}v,\ v_-:=\partial_{121}v
\end{equation}
is a homogeneous basis of $W_3$, and there is a unique scalar
$\nu\in\kk^\times$ such that
\begin{equation}\label{eq_nu}
\partial_2\partial_1\partial_2
=\nu\,\partial_1\partial_2\partial_1.
\end{equation}
Thus $A_3^+$ is a copy of the nilCoxeter algebra with the second generator
rescaled; in particular
\[
\gdim A_3^+=q^{-3}[2][3].
\]
The scalar $\nu$ cannot be changed by the common rescaling
\eqref{eq_rescale_pair}.

Homogeneity together with the relations $\partial_i\partial_i^-+\partial_i^-\partial_i=1$ imply that the
action of $\partial_1^-$ has the form
\begin{equation}\label{eq_action_minus_one}
\begin{aligned}
\partial_1^-v_-&=\partial_{21}v+\lambda_1\partial_{12}v,
&\qquad \partial_1^-\partial_1v&=v,\\
\partial_1^-\partial_{12}v&=\partial_2v+\lambda_2\partial_1v,
&\partial_1^-\partial_2v&=-\lambda_2v,\\
\partial_1^-\partial_{21}v&=\lambda_3\partial_1v-\lambda_1\partial_2v,
&\partial_1^-v&=0,
\end{aligned}
\end{equation}
and the action of $\partial_2^-$ has the form
\begin{equation}\label{eq_action_minus_two}
\begin{aligned}
\partial_2^-v_-&=\nu^{-1}\partial_{12}v+\mu_1\partial_{21}v,
&\qquad \partial_2^-\partial_1v&=-\mu_2v,\\
\partial_2^-\partial_{12}v&=\mu_3\partial_2v-\nu\mu_1\partial_1v,
&\partial_2^-\partial_2v&=v,\\
\partial_2^-\partial_{21}v&=\partial_1v+\mu_2\partial_2v,
&\partial_2^-v&=0.
\end{aligned}
\end{equation}
The six displayed parameters are not independent.

\begin{prop}\label{prop_three_strand_parameters}
With the notation above, the square-zero relations for the negative generators
are equivalent to
\begin{equation}\label{eq_parameter_constraints}
\lambda_3=-\lambda_1\lambda_2,
\qquad
\mu_3=-\nu\mu_1\mu_2.
\end{equation}
Consequently, once $\nu$ is fixed, four parameters
$\lambda_1,\lambda_2,\mu_1,\mu_2$ remain.
Moreover,
\begin{equation}\label{eq_negative_cyclic_condition}
A_3^-v_-=W_3
\quad\Longleftrightarrow\quad
(1-\nu\lambda_1\mu_1)(1-\lambda_2\mu_2)\neq0.
\end{equation}
Under the relations \eqref{eq_parameter_constraints}, one also has
\begin{equation}\label{eq_negative_braid}
\partial_2^-\partial_1^-\partial_2^-
=\nu^{-1}\partial_1^-\partial_2^-\partial_1^-.
\end{equation}
In particular, if \eqref{eq_negative_cyclic_condition} holds, then
\[
\gdim A_3^-=q^3[2][3].
\]
\end{prop}

\begin{proof}
Apply $(\partial_1^-)^2$ and $(\partial_2^-)^2$ to the basis
\eqref{eq_basis_W3}. For instance, $\pd_1^-\pd_1^-\pd_2v=\pd_1^-(-\lbd_2v)=0$ and
\[\pd_1^-\pd_1^-\pd_{21}v=\pd_1^-(\lbd_3\pd_1v-\lbd_1\pd_2v)=\lbd_3v+\lbd_1\lbd_2v.\]
The only nonzero coefficients which remain are scalar
multiples of
$\lambda_3+\lambda_1\lambda_2$ and
$\mu_3+\nu\mu_1\mu_2$, giving
\eqref{eq_parameter_constraints}.
After these substitutions, the determinant
of the six vectors
\[
v_-,\ \partial_1^-v_-,\ \partial_2^-v_-,\
\partial_1^-\partial_2^-v_-,\
\partial_2^-\partial_1^-v_-,\
\partial_1^-\partial_2^-\partial_1^-v_-
\]
with respect to the basis \eqref{eq_basis_W3} is, up to a nonzero scalar,
\[
(1-\nu\lambda_1\mu_1)^4(1-\lambda_2\mu_2)^2.
\]
This proves \eqref{eq_negative_cyclic_condition}. Finally, direct substitution
in \eqref{eq_action_minus_one}--\eqref{eq_action_minus_two} gives
\eqref{eq_negative_braid}.
\end{proof}

\begin{remark}
Hence, $A_3^+$-cyclicity (three strands) does not by itself imply $A_3^-$-cyclicity. More importantly, the parameters $\lbd_i$ and $\mu_i$ for the double nilCoxeter
algebra are not all zero, as we will see in the next subsection. Potentially, we have a 5-parameter family of monoidal categories, labeled by $(\nu,\lambda_1,\lambda_2,\mu_1,\mu_2)$, with the nonvanishing condition in \eqref{eq_negative_cyclic_condition}, but we don't investigate that multi-parameter family in the present paper.
\end{remark}

\subsection{The double nilCoxeter three-strand structure}
Consider the polynomial realization of $\wNC_3$ from
Section~\ref{doubled_nilCoxeter}. After the overall grading shift which puts the
top degree in degree $3$, take
\[
v=x_1^2x_2.
\]
Then
\[
\partial_1v=x_1x_2,\qquad
\partial_2v=x_1^2,\qquad
\partial_{12}v=x_1+x_2,\qquad
\partial_{21}v=x_1,\qquad
\partial_{121}v=1.
\]
Using $\partial_i^-=-x_i\partial_i x_{i+1}$, one obtains
\begin{equation}\label{eq_standard_negative_actions}
\begin{aligned}
\partial_1^-\partial_{121}v&=\partial_{21}v,
&\qquad
\partial_1^-\partial_{12}v&=\partial_2v+\partial_1v,\\
\partial_2^-\partial_{121}v&=\partial_{12}v-\partial_{21}v,
&\partial_2^-\partial_{12}v&=\partial_1v.
\end{aligned}
\end{equation}
Consequently, for the double nilCoxeter algebra,
\begin{equation}\label{eq_standard_parameters}
\nu=1,
\qquad
(\lambda_1,\lambda_2,\lambda_3)=(0,1,0),
\qquad
(\mu_1,\mu_2,\mu_3)=(-1,0,0).
\end{equation}
In particular, choosing the parameters $\lbd_i,\mu_i$, $1\le i \le 3$ to all be zero describes a different
three-strand structure.

It is useful to record the natural one-parameter twist of the double
nilCoxeter tower of algebras. For $\nu\in\kk^\times$, define a strict graded monoidal
category $\mcNC_\nu'$ whose endomorphism algebras are the usual $\wNC_n$ and
whose tensor product maps, on homogeneous $a\in\wNC_n$ and $b\in\wNC_m$, are
\begin{equation}\label{eq_twisted_tensor_map}
\psi_{n,m}^{\nu}(a\otimes b)
:=\nu^{-n\deg(b)/2}\,\psi_{n,m}^{1}(a\otimes b),
\end{equation}
where $\psi^1$ is the inclusion  \eqref{eq_nat_inclusions}.
Since all degrees in $\wNC_n$ are even, the exponent is integral.
Associativity follows from
\[
n\deg(b)+(n+m)\deg(c) = m\deg(c)+n(\deg(b)+\deg(c)).
\]
We write $\mcNC_\nu:=\Kar(\mcNC_\nu')$; for $\nu=1$ this recovers the
graded category from Section~\ref{subsection_minimal}. The local generators of $\mcNC_\nu$ in positions $i,i+1$ are
\begin{equation}\label{eq_twisted_local_generators}
\partial_i^{(\nu)}=\nu^{i-1}\partial_i,
\qquad
(\partial_i^-)^{(\nu)}=\nu^{1-i}\partial_i^-.
\end{equation}
Hence the three-strand parameters of $\mcNC_\nu'$ are
\begin{equation}\label{eq_nu_standard_parameters}
(\lambda_1,\lambda_2,\lambda_3)=(0,\nu,0),
\qquad
(\mu_1,\mu_2,\mu_3)=(-\nu^{-1},0,0),
\end{equation}
and the positive braid scalar is $\nu$.

\begin{definition}\label{def_nu_standard}
A three-strand structure (namely, a choice of $W_3$ and operators $\pd_1^\pm,\pd_2^\pm$ as in Section \ref{subsec:abstract 2,3 string structure}) is called \emph{$\nu$-standard} if, with the
normalization above, it is isomorphic as a graded algebra with its four
distinguished local generators to the three-strand structure of
$\mcNC_\nu'$. Equivalently, it satisfies
\eqref{eq_positive_cyclic_three}, \eqref{eq_nu}, and
\eqref{eq_nu_standard_parameters}.
\end{definition}

\subsection{Cyclicity and local generation do not imply uniqueness}

The hypotheses suggested by the first three strand calculation do not by
themselves characterize $\mcNC'$. We give an explicit second minimal
categorification.

For $n\geq0$, let
\begin{equation}\label{eq_edge_Wn}
W_n^{\mathrm{edge}}\coloneqq \bigoplus_{w\in S_n}\kk v_w,
\qquad
\deg(v_w)=\binom{n}{2}-2\ell(w).
\end{equation}
Then $\gdim W_n^{\mathrm{edge}}=[n]!$. Define degree $\pm 2$ operators by
\begin{equation}\label{eq_edge_actions}
\partial_i v_w=
\begin{cases}
v_{s_iw},&\ell(s_iw)=\ell(w)+1,\\
0,&\ell(s_iw)=\ell(w)-1,
\end{cases}
\qquad
\partial_i^- v_w=
\begin{cases}
v_{s_iw},&\ell(s_iw)=\ell(w)-1,\\
0,&\ell(s_iw)=\ell(w)+1.
\end{cases}
\end{equation}
Thus $\partial_i^-$ reverses the weak-order edge defined by $\partial_i$.
The operators of each sign satisfy the nilCoxeter relations, and
\eqref{eq_same_i_3} holds.

To define the monoidal embeddings, let $S^{n,m}$ be the set of minimal length
representatives for the left cosets
$(S_n\times S_m)\backslash S_{n+m}$. Every $w\in S_{n+m}$ has a unique
factorization
\[
w=(w_1\times w_2)u,
\qquad w_1\in S_n,\ w_2\in S_m,\ u\in S^{n,m},
\]
with $\ell(w)=\ell(w_1)+\ell(w_2)+\ell(u)$. Put
\[
M_{n,m}\coloneqq \bigoplus_{u\in S^{n,m}}\kk x_u,
\qquad
\deg(x_u)=nm-2\ell(u).
\]
The factorization above gives a graded isomorphism
\begin{equation}\label{eq_edge_factorization}
W_{n+m}^{\mathrm{edge}}
\cong W_n^{\mathrm{edge}}\otimes W_m^{\mathrm{edge}}\otimes M_{n,m},
\qquad
v_{(w_1\times w_2)u}\longmapsto v_{w_1}\otimes v_{w_2}\otimes x_u.
\end{equation}
Let
\[
\End(W_n^{\mathrm{edge}})\otimes\End(W_m^{\mathrm{edge}})
\longrightarrow\End(W_{n+m}^{\mathrm{edge}})
\]
act on the first two factors in \eqref{eq_edge_factorization} and trivially on
$M_{n,m}$. The uniqueness of parabolic factorization, applied to three
consecutive blocks, gives the associativity relations
\eqref{eq_assoc_relations}. More explicitly, let $S^{n,m,r}$ denote the set of minimal-length
representatives for the cosets
$
(S_n\times S_m\times S_r)\backslash S_{n+m+r},
$
and denote the corresponding graded multiplicity space by $M_{n,m,r}$.
The uniqueness of parabolic factorization gives an identification
\[
W_{n+m+r}^{\mathrm{edge}}
\cong
W_n^{\mathrm{edge}}\otimes W_m^{\mathrm{edge}}
\otimes W_r^{\mathrm{edge}}\otimes M_{n,m,r}.
\]
Under this identification, both iterated embeddings in
\eqref{eq_assoc_relations} act on the first three factors and act
trivially on $M_{n,m,r}$. Hence the two embeddings coincide.

We therefore obtain a graded minimal
categorification, denoted $\mcC'_{\mathrm{edge}}$.

\begin{prop}\label{prop_edge_counterexample}
The category $\mcC'_{\mathrm{edge}}$ satisfies, for every $n$,
\[
A_n^+v_\id=W_n^{\mathrm{edge}},
\qquad
A_n^-v_{w_0}=W_n^{\mathrm{edge}},
\]
and the local elements $\partial_i,\partial_i^-$ generate all of
$\End(W_n^{\mathrm{edge}})$. Nevertheless, on three strands it has
\[
\nu=1,
\qquad
\lambda_i=\mu_i=0\quad(1\leq i\leq3),
\]
so it is not graded monoidally isomorphic to $\mcNC'$.
\end{prop}

\begin{proof}
The two cyclicity assertions follow immediately from the two orientations of
weak order. For generation, set
\[
T_i\coloneqq\partial_i+\partial_i^-.
\]
Then $T_i$ is the permutation matrix for left multiplication by $s_i$ on the
basis $\{v_w\}$. Also
\[
P_i\coloneqq\partial_i^-\partial_i
\]
is the diagonal projection onto the span of those $v_w$ for which
$\ell(s_iw)=\ell(w)+1$. Conjugating the $P_i$ by the permutation matrices
$T_w$ gives the diagonal projections associated with arbitrary pairwise order
comparisons. Products of these projections isolate each one-dimensional
subspace $\kk v_w$, and the $T_w$ then produce all matrix units. Thus the local
generators generate the full matrix algebra.

On three strands, each negative operator simply reverses one positive
weak-order edge, and comparison with
\eqref{eq_action_minus_one}--\eqref{eq_action_minus_two} gives
$\lambda_i=\mu_i=0$. This differs from
\eqref{eq_standard_parameters}.
\end{proof}

Thus, even the simultaneous assumptions of positive cyclicity, negative
cyclicity, the nilCoxeter relations for both signs, local generation of every
$A_n$, and monoidal compatibility do not characterize the double nilCoxeter
categorification. In particular, a classification of arbitrary associative
systems of graded matrix algebra embeddings that categorify $\Z_q\{E\} $ or its classical version $\Z\{E\}$ seems to be an open problem.

\subsection{Three-strand rigidity}

Although the preceding hypotheses do not give uniqueness, the correct mixed
structure on three strands does.

\begin{thm}[Three-strand rigidity]\label{thm_three_strand_rigidity}
Let $(\mcC',X)$ be a strict graded minimal categorification over a field
$\kk$. If its three-strand structure is $\nu$-standard for some
$\nu\in\kk^\times$, then there is a strict graded monoidal isomorphism
\[
\mcC'\cong\mcNC_\nu'
\]
preserving the chosen generator. 
\end{thm}
 
\begin{proof}
For each $n$, let $\partial_i,\partial_i^-\in A_n$ be the local elements coming
from $A_2$. Consider the presentation of $\wNC_n$ in
Proposition~\ref{prop_wNC_def_rel}, rewritten using the rescaling
\[
\partial_i\longmapsto\nu^{i-1}\partial_i,
\qquad
\partial_i^-\longmapsto\nu^{1-i}\partial_i^-.
\]
We call this the $\nu$-twisted presentation. The distinguished local elements
of $A_n$ satisfy all relations in this presentation:
\begin{itemize}
\item the same-index relations come from the embedded copy of $A_2$;
\item every relation involving adjacent indices $i,i+1$ comes from the
corresponding embedded copy of $A_3$, which is $\nu$-standard;
\item every far-commutativity relation follows from the monoidal interchange
law, since the two morphisms live on disjoint tensor factors.
\end{itemize}
Hence Proposition~\ref{prop_wNC_def_rel} gives a unital graded algebra
homomorphism
\begin{equation}\label{eq_phi_n_rigidity}
\Phi_n:\wNC_n\longrightarrow A_n
\end{equation}
which sends the local generators of $\mcNC_\nu'$ to the local generators of
$\mcC'$.

By Proposition~\ref{prop_rho_iso},
$\wNC_n\cong\Mat_{n!}(\kk)$, so $\wNC_n$ is simple. Therefore, the unital map
\eqref{eq_phi_n_rigidity} is injective. Both algebras have dimension
$(n!)^2$, hence $\Phi_n$ is an isomorphism. 

It remains to check compatibility with the tensor product maps \eqref{eq_nm_homs}. 
For a local generator in the second tensor factor,
\eqref{eq_twisted_tensor_map} and
\eqref{eq_twisted_local_generators} give
\[
\psi_{n,m}^{\nu}
\bigl(1\otimes\partial_j^{(\nu)}\bigr)
=\partial_{n+j}^{(\nu)},
\qquad
\psi_{n,m}^{\nu}
\bigl(1\otimes(\partial_j^-)^{(\nu)}\bigr)
=(\partial_{n+j}^-)^{(\nu)}.
\]
The corresponding statement for generators in the first tensor factor is
immediate. Since the local generators generate the endomorphism algebras,
the isomorphisms $\Phi_n$ commute with all tensor-product maps. They
therefore assemble to the desired strict graded monoidal isomorphism
preserving the chosen generator.
\end{proof}

\begin{remark}\label{rmk_rigidity_assumptions}
Once the $\nu$-standard three-strand mixed structure is imposed, no assumption
of the form $A_4^+v=W_4$, $A_n^+v=W_n$ for all $n$, or
$A_3^-v_-=W_3$ is needed, and local generation of $A_n$ need not be assumed;
all of these properties follow from Theorem~\ref{thm_three_strand_rigidity}.
\end{remark}

\begin{remark}\label{rmk_nu_invariant}
The scalar $\nu$ is invariant under  graded monoidal isomorphism.
Indeed, a graded automorphism of $A_2$ can only rescale $\partial$ and
$\partial^-$ inversely as in \eqref{eq_rescale_pair}. Both copies of
$\partial$ in $A_3$ are therefore rescaled by the same scalar, leaving the
braid scalar in \eqref{eq_nu} unchanged. 
\end{remark}

\subsection{A nilHecke-realizable uniqueness criterion}

There is also a conceptual sufficient condition for uniqueness which uses the
$\slt$-action on the nilHecke algebra discussed in
Section~\ref{subsec_subalgebra}.

\begin{thm}[NilHecke-realizable uniqueness]\label{thm_nilhecke_realizable}
Let $(\mcC',X)$ be a graded minimal categorification over a field $\kk$ of characteristic $0$.
Suppose the algebras $A_n$ are equipped with compatible $\slt$-actions and
there is a faithful graded  monoidal functor
\[
F:\mcC'\longrightarrow\mcNH'
\]
taking $X$ to $\mcE$, such that every induced map
\[
F_n:A_n\longrightarrow\NH_n
\]
is $\slt$-equivariant. Then
\[
\mcC'\cong\mcNC'
\]
as a graded monoidal category.
\end{thm}

\begin{proof}
The image $F_n(A_n)$ is finite-dimensional and $\slt$-stable. Hence it lies in
the maximal finite-dimensional $\slt$-subrepresentation of $\NH_n$. By
Corollary~\ref{cor_three_coincide},
\[
F_n(A_n)\subseteq\NH_n^{\mathsf{fd}}=\wNC_n.
\]
Since $F$ is faithful, $\dim F_n(A_n)=(n!)^2$, while
$\dim\wNC_n=(n!)^2$. Therefore $F_n(A_n)=\wNC_n$ for every $n$. Monoidality of
$F$ identifies the tensor-product maps with the standard inclusions
\eqref{eq_nat_inclusions}, giving the asserted isomorphism.
\end{proof}

\begin{remark}
Theorem~\ref{thm_nilhecke_realizable} uses the realization inside the nilHecke
category in an essential way. We do not claim that the existence of an
abstract $\slt$-action on a minimal categorification, without such an
equivariant realization, implies uniqueness.
\end{remark}

%
%

\section{Minimal categorification in the odd case}\label{section-odd}

The odd nilHecke algebra $\ONH_n$, its action on skew polynomials, and its role in odd categorification have been studied in~\cite{EKL12,EllisLauda16,EllisQi16,ragavender2017odd}.  In this section we construct a semisimple finite-dimensional sub-superalgebra $\wONC_n\subset \ONH_n$, analogous to $\wNC_n\subset\NH_n$, with
\[
\wONC_n\cong \End_{\kk}(V_n).
\]
We then explain the resulting minimal monoidal supercategorification and record a natural action of the Lie superalgebra $\mathfrak{pgl}(1|1)$ on the odd nil-Hecke tower which preserves $\wONC_n$.

The ring of skew polynomials
\[\SPol_n = \kk\langle x_1,\ldots,x_n \rangle / \langle x_i x_j + x_j x_i = 0 \text{ for } i\neq j \rangle\]
is a $\Z \times \Z/2$-graded superalgebra with $\Z$-grading $|\cdot|$ and $\Z/2$-grading $p(\cdot)$, determined by $|x_i|=2$ and $p(x_i)=1$. 
For a homogeneous element $f$, we then have $p(f)\equiv |f|/2\pmod 2$. Following~\cite[Section~2.1]{EKL12}, let $S_n$ act on $\SPol_n$ by the signed permutation action
\[s_i(x_j)=\begin{cases}
    -x_{i+1} & j=i,\\
    -x_i & j=i+1,\\
    -x_j & \te{else}
\end{cases}.\]

The odd divided difference operator $\partial_i$ is characterized by
\[
\partial_i(x_j)=
\begin{cases}
1,&j=i,i+1,\\
0,&j\neq i,i+1,
\end{cases}
\]
and the Leibniz rule
\begin{equation}\label{eq_odd_Leibniz}
\partial_i(fg)=\partial_i(f)g+s_i(f)\partial_i(g).
\end{equation}
Note that the ``odd twist'' in this Leibniz rule is hidden in the signed action of the symmetric group. In particular, $|\partial_i|=-2$ and $p(\partial_i)=1$. The operator $\pd_i$ also has the explicit formula 
\[\partial_i(f)=\frac{(x_{i+1}-x_i)f-s_i(f)(x_{i+1}-x_i)}{x_{i+1}^2-x_i^2}.\]

The odd nilHecke algebra $\ONH_n$ is defined as the sub-superalgebra of
$\End_{\kk}(\SPol_n)$ generated by multiplications by the $x_i$'s and by the
$\partial_i$'s.   The generators have degrees $|x_i|=2$, $|\partial_i|=-2$, $p(x_i)=p(\partial_i)=1$.

 A set of defining relations for $\ONH_n$ is given by 
\begin{equation}\label{eq_nilHecke_def_odd}
 \begin{array}{ll}
   \partial_i x_j + x_j\partial_i=0 \quad \text{if $j\neq i,i+1$}, &
   \partial_i\partial_j + \partial_j\partial_i =0 \quad \text{if $|i-j|>1$}, \\
  \partial_i^2 = 0,  &
   \partial_i\partial_{i+1}\partial_i = \partial_{i+1}\partial_i\partial_{i+1},  \\
   x_i \partial_i + \partial_i x_{i+1}=1,  &   \partial_i x_i + x_{i+1} \partial_i =1, \\
   x_i x_j +   x_j x_i  = 0 \quad \text{if $i\not= j$}. &  
  \end{array}
\end{equation}
These are the relations of~\cite[Proposition~2.1]{EKL12}.

The odd symmetric polynomials form the sub-superalgebra
\[
\OLbd_n\coloneqq \bigcap_{i=1}^{n-1}\ker(\partial_i)\subset \SPol_n.
\]
If $h\in\OLbd_n$, then~\eqref{eq_odd_Leibniz} gives 
$\partial_i(fh)=\partial_i(f)h$, so the action of $\ONH_n$ on $\SPol_n$ is right $\OLbd_n$-linear.  We will use the results of~\cite[Proposition~2.13 and Corollary~2.14]{EKL12} that $\SPol_n$ is free as a right $\OLbd_n$-module and that
\begin{equation}\label{eq_odd_matrix_over_sym}
\ONH_n\cong \End_{\OLbd_n}(\SPol_n).
\end{equation}

Set
\begin{equation}\label{eq_undxn} 
\underline{x}_n=x_1^{n-1}x_2^{n-2}\cdots x_{n-1}
\end{equation}
and consider the localization
\[
\SPol_n^{\circ}=\kk\langle x_1^{\pm1},\ldots,x_n^{\pm1}\rangle/
\langle x_i^{\epsilon}x_j^{\epsilon'}+x_j^{\epsilon'}x_i^{\epsilon}=0
\text{ for }i\neq j,\ \epsilon,\epsilon'\in\{\pm1\}\rangle.
\]
Let $\SPol_n^-$ be the subalgebra of $\SPol_n^{\circ}$ generated by $x_1^{-1},\ldots,x_n^{-1}$, let $\SPol_n^+$ be the subalgebra generated by $x_1,\ldots,x_n$ which is equal to $\SPol_n$, and let
$\psi$ be the algebra involution of $\SPol_n^{\circ}$ determined by
$\psi(x_i)=x_i^{-1}$. 
We also set $\opartial_i = \psi \partial_i \psi$, so that

\[
\overline{\partial}_i(f)=\frac{(x_{i+1}^{-1}-x_i^{-1})f-s_i(f)(x_{i+1}^{-1}-x_i^{-1})}{x_{i+1}^{-2}-x_i^{-2}}.
\]

\begin{remark}
It is not true that $\ol\pd_i=\pm x_ix_{i+1}\pd_i$, which holds in the even case \eqref{eq_opart_t}. However, the odd analogue of the formula $\pd_i^-=\tau \pd_i\tau$ can still be derived, see \eqref{eq_odd_dminus_def} and Proposition~\ref{prop_odd_conjugation} below. Note that $\overline{\partial}_i$ does not belong to $\ONH_n$ -- for instance, one can compute that $\ol\pd_1(x_1)=-x_1x_2$, $\ol\pd_1(x_2)=x_1x_2$, and $\ol\pd_1(1)=0$, so $\ol\pd_1$ should be a degree 2 map. Yet degree 2 elements in $\ONH$ have a basis of  $x_1,x_2,x_1^2\pd_1,x_1x_2\pd_1,x_2^2\pd_1$; no linear combination of these elements agrees with $\ol\pd_1$. 
\end{remark}

Define the negative odd divided difference operator (of degree $2$ and odd parity) 
\begin{equation}\label{eq_odd_dminus_def}
\partial^-_i \coloneqq x_{i}\partial_ix_{i+1}.
\end{equation}
We define $\wONC_n$ to be the algebra generated by $\partial_i, \partial^-_i$ for $i=1,\ldots,n-1$.

Define $\tau\colon \SPol_n^\circ\lto\SPol_n^\circ$ via $f \lmto \psi(f)\underline{x}_n$. We have 
\[\tau^2(f) = \psi(\psi(f)\underline{x}_n))\underline{x}_n=f\psi(\underline{x}_n)\underline{x}_n=(-1)^{\binom{\lfloor{\frac{n}{2}\rfloor}}{2}}f,\]
where $\ul x_n$ is given by~\eqref{eq_undxn}. Therefore $\tau^{-1}=(-1)^{\binom{\lfloor{\frac{n}{2}\rfloor}}{2}}\tau$.

\begin{remark}
    In the above, we have $\tau^2=(-1)^{\binom{\lfloor n/2\rfloor}{2}}\id$, because $\psi\pr{\ul x_n}\ul x_n=(-1)^{\binom{\lfloor n/2\rfloor}{2}}$; for this to be true, we want $\psi$ to be an involution.
\end{remark} 

\begin{prop}\label{prop_odd_conjugation}
On $\SPol_n^{\circ}$ we have
\begin{equation}\label{eq_odd_conjugation_revised}
\partial_i^-=\tau^{-1}\partial_i\tau.
\end{equation}
\end{prop}
\begin{proof}
    Note that
    \begin{align*}
        \ol\pd_i(f)&=\frac{(x_{i+1}^{-1}-x_i^{-1})f-s_i(f)(x_{i+1}^{-1}-x_i^{-1})}{x_{i+1}^{-2}-x_i^{-2}}\\
        &=\frac{(x_ix_{i+1}^2-x_i^2x_{i+1})f-s_i(f)(x_ix_{i+1}^2-x_i^2x_{i+1})}{x_{i+1}^2-x_i^2}\\
        &=\frac{(x_{i+1}^2-x_i^2)x_if-x_i^2((x_{i+1}-x_i)f-s_i(f)(x_{i+1}-x_i))-(x_{i+1}^2-x_i^2)s_i(f)x_i}{x_{i+1}^2-x_i^2}\\
        &=\frac{(x_{i+1}^2-x_i^2)x_if-(x_{i+1}^2-x_i^2)x_i^2\pd_i(f)-(x_{i+1}^2-x_i^2)s_i(f)x_i}{x_{i+1}^2-x_i^2}\\
        &=x_if-x_i^2\pd_i f-s_i(f)x_i.
    \end{align*}
    Note $\pd_i^-(f)=x_i\pd_i(x_{i+1}f)=x_i(f-x_i\pd_i(f))=x_if-x_i^2\pd_i(f)$, so that
    \[\pd_i^-(f)=\ol\pd_i(f)+s_i(f)x_i.\]
    On the other hand, as $\tau^{-1}(f)=\psi(f)\psi(\ul x_n)^{-1}$, we have
    \begin{align*}
        \tau^{-1}\pd_i\tau(f)&=\psi(\pd_i(\psi(f)\ul x_n))\psi(\ul x_n)^{-1}\\
        &=\ol\pd_i(f)+s_i(f)\psi(\pd_i \ul x_n)\psi(\ul x_n)^{-1}\\
        &=\ol\pd_i(f)+s_i(f)x_i;
    \end{align*}
    Hence $\pd_i^-=\tau^{-1}\pd_i\tau$, as claimed.
\end{proof}

Define
\[V_n\coloneqq \SPol_n^+\cap \SPol_n^- \ul x_n\subset\SPol_n^{\circ}.\] 
\begin{prop}\label{prop_odd_tau_end}
    $\tau$ restricts to an endomorphism on $V_n$.
\end{prop}
\begin{proof}
    For the same reasons as in the even case, $V_n$ has a basis given by
    \[\{x_1^{a_1}\cdots x_{n-1}^{a_{n-1}}:0\le a_i\le n-i\}.\] 
    Then we have
    \[\tau(x_1^{a_1}\cdots x_{n-1}^{a_{n-1}})=x_1^{-a_1}\cdots x_{n-1}^{-a_{n-1}}\cdot x_1^{n-1}\cdots x_{n-1}^1=\pm x_1^{n-1-a_1}x_2^{n-2-a_2}\cdots x_{n-1}^{1-a_{n-1}},\] 
    where the sign is determined by the odd commutation rule. Again, $n-i\ge n-i-a_i\ge 0$, so this is still in $V_n$. 
\end{proof}

Let $\ONC_n^\pm$ denote the subalgebra of endomorphisms of $\SPol_n^\circ$ generated by $\pd_1^\pm,\cdotsc,\pd_{n-1}^\pm$. Since we have seen that $\pd_i^-=\tau^{-1}\pd_i \tau$, we obtain the following:
\begin{cor}\label{cor_odd_plus_minus_iso}
    There is a natural isomorphism $\ONC_n^+\simlto \ONC_n^-$ taking $\pd_i$ to $\pd_i^-$, and the latter satisfy the relations 
    \begin{align*}
        \pd_i^-\pd_i^-&=0,\\
        \pd_i^-\pd_{i+1}^-\pd_i^-&=\pd_{i+1}^-\pd_i^-\pd_{i+1}^-,\\
        \pd_i^-\pd_j^-&=-\pd_j^-\pd_i^-\quad\te{for }|i-j|>1.
    \end{align*}
\end{cor}
\begin{lem}\label{lem_odd_V_stable}
    $V_n$ is stable under the action of $\pd_i$ and $\pd_i^-$. 
\end{lem}
\begin{proof}
    This follows immediately from Propositions \ref{prop_odd_conjugation} and \ref{prop_odd_tau_end}.
\end{proof}
We can hence let $\wh\ONC_n$ be the subalgebra of endomorphisms of $\SPol_n$ generated by $\pd_i$ and $\pd_i^-$ for $i=1,\cdotsc,n-1$: 
\[\wh\ONC_n\coloneqq \wan{\pd_1,\cdotsc,\pd_{n-1},\pd_1^-,\cdotsc,\pd_{n-1}^-}\subset\ONH_n.\]
This then acts on $V_n$:
\[\rho\colon \wh\ONC_n\lto \End_\bk (V_n).\] 
Proposition 2.13 in~\cite{EKL12} shows that $\SPol_n$ is a free $\OLbd_n$-module generated by $V_n$, and Proposition 2.11 with Corollary 2.14 \textit{ibid.} shows that the action of $\ONH_n$ on $\SPol_n$ is faithful and gives an isomorphism
\[
\ONH_n \cong \End_{\mathsf{O}\Lambda_n}(\SPol_n) \cong \Mat_{n!}(\mathsf{O}\Lambda_n). 
\]
Hence, just as in the even case \ref{cor_rho_injective}, we have
\begin{cor}\label{cor_odd_rho_injective}
    The map $\rho\colon \wh\ONC_n\lto \End_\bk(V_n)$ above is injective. 
\end{cor}

\begin{prop}\label{prop_odd_relations}
    The following relations hold in $\wONC_n$:
    \begin{eqnarray}
    \pd_i\pd_j^-&=&-\pd_j^-\pd_i\qquad\te{if }|i-j|>1,\\
    1&=& \pd_i\pd_i^-+\pd_i^-\pd_i,\\
    \pd_i\pd_{i+1}^-&=&\pd_i^-\pd_{i+1}-\pd_i^-\pd_i+1-\pd_{i+1}^-\pd_{i+1},\\
    \pd_{i+1}\pd_i^-&=&-\pd_i^-\pd_{i+1}+[\pd_i^-,\pd_{i+1}^-]_{\mathrm s}[\pd_i,\pd_{i+1}]_{\mathrm s}.
  \end{eqnarray}
\end{prop}
These relations allow to move $\partial_i$'s to the right of $\partial_j^-$'s, and vice versa. 
 Here the bracket $[x,y]_{\mathrm s}$ refers to the super bracket, so that for instance $[\pd_i^\pm,\pd_{i+1}^\pm]_{\mathrm s}=\pd_i^\pm\pd_{i+1}^\pm+\pd_{i+1}^\pm\pd_i^\pm$. Note that the last equation has one fewer term than in the even case, in Proposition \ref{prop_relations}. Three out of these four relations can be fully rewritten using supercommutators: 
 \[
 [\partial_i, \partial_j^-]_{\mathrm s}=0\uad (\te{for }|i-j|>1), \ \ [\partial_i,\partial_i^-]_{\mathrm s}=1, \ \ [\pd_{i+1},\pd_i^-]_{\mathrm s}=[\pd_i^-,\pd_{i+1}^-]_{\mathrm s}[\pd_i,\pd_{i+1}]_{\mathrm s},
 \]
 and the remaining relation can be modified to 
\[
\pd_i(\pd_{i+1}^- - \pd_i^-) =(\pd_i^- -\pd_{i+1}^-)\pd_{i+1}.
\]
or, equivalently, to 
\[
    \pd_i^-(\pd_{i+1}-\pd_i)=(\pd_i-\pd_{i+1})\pd_{i+1}^-.
\]

\begin{proof}
    It is worth noting that in the odd case, we have $\pd_1x_1\pd_1=\pd_1x_2\pd_1=\pd_1$. The first two equations are easy to check. 
    The third equation can be checked by noting (here again we restrict to 3 strings, denoting $x_1=x,x_2=y,x_3=z$)
    \begin{align*}
        \pd_1\pd_2^-&=-x\pd_1\pd_2z+\pd_2z\\
        &=x\pd_1y\pd_2-x\pd_1+\pd_2z\\
        &=x\pd_1y\pd_2-x\pd_1y\pd_1+1-y\pd_2\\
        &=x\pd_1y\pd_2-x\pd_1y\pd_1+1-y\pd_2z\pd_2.
    \end{align*}

    The fourth equation can be checked by computing the following:
    \begin{align*}
        \pd_2\pd_1^-&=-x^2\pd_2\pd_1-x\pd_2,\\
        \pd_1^-\pd_2&=-x^2\pd_1\pd_2+x\pd_2,\\
        \pd_1^-\pd_2^-\pd_1\pd_2&=-x^3\pd_{1,2,1}+x^2y\pd_{1,2,1}-xy^2\pd_{1,2,1}-x^2\pd_1\pd_2+xy\pd_1\pd_2,\\
        \pd_2^-\pd_1^-\pd_1\pd_2&=xy^2\pd_{1,2,1}-xy\pd_1\pd_2,\\
        \pd_1^-\pd_2^-\pd_2\pd_1&=x^3\pd_{1,2,1}-x^2\pd_2\pd_1+xy\pd_2\pd_1,\\
        \pd_2^-\pd_1^-\pd_2\pd_1&=-x^2y\pd_{1,2,1}-xy\pd_2\pd_1,
    \end{align*}
    so that one may check directly that
    \[\pd_2\pd_1^-=-\pd_1^-\pd_2+\pd_1^-\pd_2^-\pd_1\pd_2+\pd_2^-\pd_1^-\pd_1\pd_2+\pd_1^-\pd_2^-\pd_2\pd_1+\pd_2^-\pd_1^-\pd_2\pd_1.\] 
\end{proof}

Let us define, as in \cite{EKL12}, the odd Schubert polynomials as
\[\Slie_w\coloneqq \pd_{w^{-1}w_0}(\ul x_n),\]
as well as the negative odd Schubert polynomials as 
\[\Slie_w^-\coloneqq \pd_w^-(1).\] 
Note that $\deg\Slie_w=\deg\Slie_w^-=2\ell(w)$.
Lemma 2.12 in~\cite{EKL12} tells us that the odd Schubert polynomials form a basis for $V_n$.   As in the even case, observing that
\[\Slie_w^-=\pd_w^-(1)=\tau^{-1}\pd_w \tau (1)=\tau^{-1}\pd_w \ul x_n=\tau^{-1}\Slie_{w_0w^{-1}}\]
and that $\tau$ is (up to sign) an involution on $V_n$, we deduce that the negative Schubert polynomials $\{\Slie_w^-\}_w$ also give a basis of $V_n$. We may then prove the following in much the same way as in the even case.
\begin{prop}\label{prop_odd_rho_iso}
   The map $\rho\colon \wh\ONC_n\simlto \End_\bk(V_n)$ is an isomorphism of algebras. 
\end{prop}
\begin{proof}
    We have seen earlier that the map is injective. \cite[Equation (2.42)]{EKL12} shows that
    \[\pd_u\Slie_w=\begin{cases}
        \pm\Slie_{wu^{-1}} & \ell(wu^{-1})=\ell(w)-\ell(u),\\
        0 &\te{else}
    \end{cases}.\]
    Hence the proof in the even case (Proposition \ref{prop_rho_iso}) works again, namely a triangularity argument with respect to the bases of Schubert polynomials and negative Schubert polynomials. The same argument goes through.
\end{proof}

\begin{cor}\label{cor_odd_def_relations}
    A full set of defining relations in $\wh\ONC_n$ is given by the relations in Proposition \ref{prop_odd_relations}, Corollary~\ref{cor_odd_plus_minus_iso}, and the usual relations on $\pd_i$'s, namely
    \[
    \pd_i^2=0,\qquad \pd_i\pd_{i+1}\pd_i=\pd_{i+1}\pd_i\pd_{i+1},\qquad \pd_i\pd_j=-\pd_j\pd_i \  \mathrm{for} \ |i-j|>1. 
    \]
\end{cor}
\begin{proof}
    The proof is identical to that of the even case, namely Proposition \ref{prop_wNC_def_rel}. 
\end{proof}

\subsection{Monoidal product}\label{subsect:odd monoidal}
It is worth noting that, similarly to Section \ref{subsection_minimal}, the natural inclusion of algebras $\ONH_n\otimes\ONH_m\lto \ONH_{n+m}$ induces natural inclusions $\ONC^\eps_n\otimes\ONC^\eps_m\lto\ONC^\eps_{n+m}$ for $\eps\in\{+,-\}$, which gives natural inclusions
\begin{equation}\label{eq:oddmonoidal}
\wh\ONC_n\otimes\wh\ONC_m\lto\wh\ONC_{n+m},
\end{equation} 
where now $\otimes$ denotes the super tensor product. This map takes $\pd_i^\eps\otimes 1\lmto \pd_i^\eps$ and $1\otimes\pd_j^\eps\lmto\pd_{n+j}^\eps$. 

Just as in Section \ref{subsection_minimal}, we may consider a $\bk$-linear monoidal supercategory $\cal{ONC}'$ with a generating object $\CE_0$ and
\[\Hom_{\cal{ONC}'}(\CE_0^{\otimes n},\CE_0^{\otimes m})=\begin{cases}
    \wh\ONC_n & n=m\\0 &\te{else}
\end{cases}.\]
The tensor product is given on morphism via the maps \eqref{eq:oddmonoidal}. Let $\cal{ONC}=\Kar(\cal{ONC}')$ be the Karoubi envelope. There is also an odd nil-Hecke category $\cal{ONH}$, namely the Karoubi envelope of the strict monoidal $\bk$-linear category $\cal{ONH}'$ with generating object $\CE$ such that 
\[\hom_{\cal{ONH}'}(\CE^{\otimes n},\CE^{\otimes m})=\begin{cases} \ONH_n &\te{if }n=m\\ 0 &\te{else}\end{cases}.\]
Again the inclusion $\wh\ONC_n\subset\ONH_n$ induces a natural functor $\cal{ONC}\lto\cal{ONH}$, and by Proposition \ref{prop_odd_rho_iso}, this functor induces an isomorphism on the Grothendieck ring after forgetting the parity grading:
\[K_0(\cal{ONC})\cong K_0(\cal{ONH})\cong\BZ\{E\}.\]

One could also keep track of the parity grading, as is the case in, for instance, \cite{hill2015categorification} and \cite{EllisLauda16}. Let $K_0^\pi$ denote the ``parity Grothendieck group'', which is a module over the ring $\CA=\BZ[q,q^{-1},\pi]/(\pi^2-1)$, such that $[\Pi M]=\pi[M]$. In \cite{hill2015categorification} it was shown that 
\[K_0^\pi(\cal{ONH})\cong {U}_{q,\pi}^+(\sl_2),\] 
where ${U}_{q,\pi}^+(\sl_2)$ is the parity integral form of the nilpotent part of quantum group considered in \cite{hill2015categorification,EllisLauda16}, spanned over the ring $\CA$ by the $E^{(n)}$'s. Again by Proposition \ref{prop_odd_rho_iso}, and since $V_n$ is a $(\BZ\times\BZ/2)$-graded simple of $\ONH_n$ (namely the simple head of the projective representation), we have an induced isomorphism on the parity Grothendieck ring:
\[K_0^\pi(\cal{ONC})\cong K_0^\pi(\cal{ONH})\cong {U}_{q,\pi}^+(\sl_2).\]

\subsection{A \texorpdfstring{$\pgl(1|1)$}{pgl(1,1)}-action}
Let $\on{char}\bk\neq 2$. Recall that $\gl(1|1)$ is the Lie superalgebra of $2\times 2$ matrices, where the diagonal matrices are in even grading and the off-diagonal matrices are in odd grading, and the Lie superbracket is  $[x,y]=xy-(-1)^{p(x)p(y)}yx$. It has a homogeneous basis 
\[e=\tbt{0}{1}{0}{0},\qquad c=\tbt{1}{0}{0}{1},\qquad h=\tbt{1}{0}{0}{-1},\qquad f=\tbt{0}{0}{1}{0},\]
where $p(c)=p(h)=0$ and $p(e)=p(f)=1$. 
The center of this algebra is $\kk c$. 

$\pgl(1|1)$ is then the Lie superalgebra obtained from $\gl(1|1)$ by quotienting out the center. It is generated by $e,h,f$, with commutators 
\[[h,e]=2e,\qquad [h,f]=-2f,\qquad [e,e]=[e,f]=[f,f]=0.\] 

\begin{prop}\label{prop:odd pgl action}
    There is a $\pgl(1|1)$-action by super-derivations on $\ONH_n$ given by
    \begin{align*}
        e(x_i)&=x_i^2,  \qquad e(\pd_i)=1,\\
        h(a)&=\deg(a)a, \\
         f(x_i)&=1,\qquad\   f(\pd_i)=0.
    \end{align*}
\end{prop}
\begin{proof}
    By \cite[Corollary 3.9]{EllisQi16}, the  action of $e$ is well-defined. The action of $h$ is also well-defined as the relations of $\ONH_n$ are homogeneous. It is straightforward to check that the action of $f$ is well-defined and to check the $\pgl(1|1)$ relations. 
\end{proof}
This action restricts to $\wh\ONC_n$. 
\begin{prop}
    Under the action of Proposition \ref{prop:odd pgl action}, the subalgebra $\wh\ONC_n$ is a $\pgl(1|1)$-submodule.
\end{prop}
\begin{proof}
    We claim that
    \begin{equation*}
        e(\pd_i^-)=0,\ \
        h(\pd_i^-)=2\pd_i^-,\ \
        f(\pd_i^-)=1.
    \end{equation*}
    Indeed, using that $\pd_i^-=x_i\pd_ix_{i+1}$, 
    \begin{align*}
        e(x_i\pd_ix_{i+1})&=e(x_i)\pd_ix_{i+1}-x_ie(\pd_i)x_{i+1}+x_i\pd_ie(x_{i+1})\\
        &=x_i^2\pd_ix_{i+1}-x_ix_{i+1}+x_i\pd_ix_{i+1}^2\\
        &=x_i(x_i\pd_i+\pd_ix_{i+1})x_{i+1}-x_ix_{i+1}\\
        &=0,\\
        h(x_i\pd_ix_{i+1})&=2\pd_i^-,\\
        f(x_i\pd_ix_{i+1})&=f(x_i)\pd_ix_{i+1}-x_if(\pd_i)x_{i+1}+x_i\pd_if(x_{i+1})\\
        &=\pd_ix_{i+1}+x_i\pd_i\\
        &=1.
    \end{align*}
\end{proof}

    Unlike the even case, where $\wNC_n$ is the maximal $\slt$-finite submodule of $\NH_n$, the entire $\ONH_n$  
is a locally finite $\pgl(1|1)$-module.

\bibliographystyle{plain}
\bibliography{bibliography}
\end{document}